\documentclass[reqno,11pt]{amsart}
\usepackage{amsthm,amsfonts,amssymb,euscript,mathrsfs,graphics,color,amsmath,latexsym,marginnote,enumitem}

\usepackage[dvipsnames]{xcolor}

\usepackage{hyperref}
\hypersetup{colorlinks=true, pdfstartview=FitV, linkcolor=BrickRed,citecolor=BrickRed, urlcolor=BrickRed}
\definecolor{labelkey}{rgb}{0.6,0,0}

\usepackage{graphicx}

\numberwithin{equation}{section}

\providecommand{\ip}[1]{\langle#1\rangle}
\providecommand{\abs}[1]{\left\lvert#1\right\rvert}
\providecommand{\norm}[1]{\left\|#1\right\|}

\def\C{{\mathbb{C}}}

\def\varep{\varepsilon}

\def\bar{\overline}
\def\R{{\mathbb R}}

\def\R{{\bf R}}

\def\Z{{\mathbb Z}}

\def\bar{\overline}

\def\R{\mathbb{R}}

\def\T{{\mathbb T}}

\newtheorem{theorem}{Theorem}[section]
\newtheorem{lemma}[theorem]{Lemma}
\newtheorem{proposition}[theorem]{Proposition}

\newtheorem{definition}[theorem]{Definition}
\newtheorem{remark}[theorem]{Remark}

\begin{document}

\title[The Vlasov-Poisson system around homogeneous equilibria]{On the stability of homogeneous equilibria in the Vlasov-Poisson system on $\R^3$}

\author{A.\ D.\ Ionescu}
\address{Princeton University}
\email{aionescu@math.princeton.edu}
\author{B.\ Pausader}
\address{Brown University}
\email{benoit\_pausader@brown.edu}
\author{X. Wang}
\address{YMSC, Tsinghua University \& BIMSA}
\email{xuecheng@tsinghua.edu.cn}
\author{K.\ Widmayer}
\address{University of Zurich \& University of Vienna}
\email{klaus.widmayer@math.uzh.ch}

\thanks{A.I.\ was supported in part by NSF grant DMS-2007008. B.P.\ was partially supported by a Simons Fellowship, the CY-IAS fellowship program, and NSF grant DMS-2154162. X.W.\  was supported in part by NSFC-12141102, and MOST-2020YFA0713003. K.W.\ gratefully acknowledges support of the SNSF through grant PCEFP2\_203059.}

\begin{abstract}

The goal of this article is twofold. First, we investigate the linearized Vlasov-Poisson system around a family of spatially homogeneous equilibria in $\mathbb{R}^3$ (the unconfined setting). Our analysis follows classical strategies from physics \cite{BiTr,Lan1946,Pe} and their subsequent mathematical extensions \cite{BeMaMoLin,Deg,GlSh,GrNgRo,HKNgRo,MouVil}. The main novelties are a unified treatment of a broad class of analytic equilibria and the study of a class of generalized Poisson equilibria. For the former, this provides a detailed description of the associated Green's functions, including in particular precise dissipation rates (which appear to be new), whereas for the latter we exhibit explicit formulas. 

Second, we review the main result and ideas in our recent work \cite{IoPaWaWiPoisson} on the full global nonlinear asymptotic stability of the Poisson equilibrium in $\R^3$.

\end{abstract}
\maketitle

\setcounter{tocdepth}{1}

\tableofcontents

\section{Introduction}

In this article we consider a hot, unconfined, electrostatic plasma of electrons on a uniform, static, background of ions, in three dimensions. Here collisions are neglected and the associated distribution of electrons is modeled by a function $F=F(x,v,t):\mathbb{R}^3_x\times\mathbb{R}^3_v\times[0,T]\to\R$ satisfying the nondimensionalized Vlasov-Poisson system
\begin{equation}\label{VPO}
\begin{split}
\left(\partial_t+v\cdot\nabla_x\right)F+\nabla_x\phi\cdot\nabla_vF=0,\qquad \Delta_x\phi=\rho=\int_{\mathbb{R}^3}Fdv-1.
\end{split}
\end{equation}
Our goal is to investigate the stability properties of solutions around a {\it spatially homogeneous equilibrium} $M_0=M_0(v)$ normalized such that $\int_{\R^3}M_0(v)\,dv =1$. Looking for solutions of \eqref{VPO} of the form $F=M_0+f$, we obtain the system
\begin{equation}\label{VP}
\begin{split}
&\left(\partial_t+v\cdot\nabla_x\right)f+E\cdot\nabla_vM_0+E\cdot\nabla_vf=0,\\
&E:=\nabla_x\Delta_x^{-1}\rho,\qquad \rho(x,t):=\int_{\mathbb{R}^3}f(x,v,t)dv.
\end{split}
\end{equation}

To recast this as a problem for the density $\rho$, we introduce the ``backwards characteristics'' of the system\eqref{VP}: these are the functions $X,V:\R^3\times\R^3\times\mathcal{I}^2_T\to\R^3$ obtained by solving the ODE system
\begin{equation}\label{Lan1}
\begin{alignedat}{2}
&\partial_sX(x,v,s,t)=V(x,v,s,t),\qquad &X(x,v,t,t)&=x,\\
&\partial_sV(x,v,s,t)=E(X(x,v,s,t),s),\qquad &V(x,v,t,t)&=v,
\end{alignedat}
\end{equation}
where $\mathcal{I}^2_T:=\{(s,t)\in[0,T]^2:\,s\leq t\}$. The main equation \eqref{VP} gives
\begin{equation*}
\begin{split}
\frac{d}{ds}f(X(x,v,s,t),V(x,v,s,t),s)
&=-E(X(x,v,s,t),s)\cdot\nabla_vM_0(V(x,v,s,t)).
\end{split}
\end{equation*}
Integrating over $s\in[0,t]$ we have
\begin{equation}\label{Lan2}
\begin{split}
f(x,v,t)=f_0(X(x,v,0,t),V(x,v,0,t))
-\int_0^tE(X(x,v,s,t),s)\cdot \nabla_v M_0(V(x,v,s,t))\,ds,
\end{split}
\end{equation}
for any $(x,v,t)\in\R^3\times\R^3\times[0,T]$. 



\subsection{Linear analysis}\label{LinearAnalysis}
Linearizing \eqref{Lan2} for small $f$ we obtain
\begin{equation}
\begin{split}
f_{lin}(x,v,t)&:=f_0(x-tv,v)-\nabla_vM_0(v)\cdot\int_{s=0}^t E_{lin}(x-(t-s)v,s)ds,
\end{split}
\end{equation}
which upon integration yields an equation for the spatial density:
\begin{equation}\label{LinVolterra}
\begin{split}
\rho_{lin}(x,t)+\int_{\tau=0}^t\int_{\mathbb{R}^3}(t-\tau)M_0(v)\rho_{lin}(x-(t-\tau)v,\tau)d\tau dv&=\int_{\mathbb{R}^3}f_0(x-tv,v)dv.
\end{split}
\end{equation}
Motivated by this 
we consider the general Volterra equation for $(x,t)\in\R^3\times[0,\infty)$:
\begin{equation}\label{PreVolterra}
\begin{split}
\rho(x,t)+\int_{0}^t\int_{\mathbb{R}^3}(t-\tau)M_0(v)\rho(x-(t-\tau)v,\tau)d\tau dv&=h(x,t).
\end{split}
\end{equation}
This can be easily studied in the Fourier space: taking the Fourier transform in $x$, we get
\begin{equation}\label{Volterra}
\begin{split}
\widehat{\rho}(\xi,t)+\int_{0}^t(t-\tau)\widehat{\rho}(\xi,\tau)\widehat{M}_0((t-\tau)\xi)d\tau=\widehat{h}(\xi,t).
\end{split}
\end{equation}
Following classical references, like \cite{Pe}, \cite[Chapter III.30]{LLX1981} or \cite[Chapter 5.2.4]{BiTr}, we define
\begin{equation}\label{more45}
\begin{split}
&R(\xi,\theta):=\int_{0}^\infty\widehat{\rho}(\xi,t)e^{-it\theta}dt,\\
&K(\xi,\theta):=\int_{0}^\infty t\widehat{M}_0(t\xi)e^{-it\theta}dt,\\
&H(\xi,\theta):=\int_{0}^\infty \widehat{h}(\xi,t)e^{-it\theta}dt,
\end{split}
\end{equation}
for $\theta\in\mathbb{R}$. We obtain from \eqref{Volterra} the key algebraic equation
\begin{equation}\label{AE}
\begin{split}
(1+K(\xi,\theta))R(\xi,\theta)=H(\xi,\theta),
\end{split}
\end{equation}
which can easily be inverted so long as $1+K$ does not vanish (the function $K$ is the Fourier transform of the ``polarization function"). In this case, inverting the Fourier transform, we obtain a solution formula for \eqref{Volterra}. More precisely, if $\rho$ solves the equation \eqref{PreVolterra}, then we can reconstruct $\rho$ from the forcing term $h$ according to the formulas
\begin{equation}\label{SolvingVolterra1}
\widehat{\rho}(\xi,t)=\int_{0}^t \widehat{G}(\xi,t-\tau)\widehat{h}(\xi,\tau)d\tau,\qquad \rho(x,t)=\int_{0}^t\int_{\mathbb{R}^3}G(x-y,t-s)h(y,s)\,dyds,
\end{equation}
where, with $\mathfrak{1}_{[0,\infty)}$ denoting the characteristic function of the interval $[0,\infty)$,
\begin{equation*}
\widehat{G}(\xi,\tau):=\frac{1}{2\pi}\mathfrak{1}_{[0,\infty)}(\tau)\int_{\R}\frac{e^{i\theta\tau}d\theta}{1+K(\xi,\theta)}.
\end{equation*}
These formulas are the key to understanding solutions of the Vlasov-Poisson system. 
Our first main goal in what follows is to get a good description of Green's function $G$.

\subsubsection{Acceptable equilibria and Green's function}

We consider a class of radially symmetric equilibria $M_0(v)$ such that $\int_{\R^3} M_0dv=1$. We define $m_0:\mathbb{R}\to\mathbb{R}$ such that
\begin{equation}\label{m0}
m_0(r)=\int_{\mathbb{R}^2}M_0(r,y)dy,
\end{equation}
where $r\in\R$, $y\in\R^2$, $(r,y)\in\R^3$. The reduced distribution $m_0$ is even, integrable and 
\begin{equation*}
\begin{split}
\widehat{M_0}(\xi)=\widehat{m_0}(\vert \xi\vert).
\end{split}
\end{equation*}
For any $\vartheta\in(0,1/2]$ we define the regions
\begin{equation}\label{more10}
\begin{split}
&\mathcal{D}_{\vartheta}:=\{z\in\C:\,|\Im z|<\vartheta(1+|\Re z|)\},\\
&\Gamma^{+}_\vartheta:=\{z=x+i\vartheta(1+x):\,x\in[0,\infty)\},\qquad \Gamma^{-}_\vartheta:=\{z=x+i\vartheta(1-x):\,x\in(-\infty,0]\}.
\end{split}
\end{equation}

\begin{definition}\label{DefinitionEquilibrium}

We define the class $\mathcal{M}_{\vartheta,d}$, $\vartheta\in(0,\pi/4]$, $d>1$, of ``acceptable equilibria" as the set of radially symmetric functions $M_0:\R^3\to\R$ satisfying $\int_{\R^3} M_0dv=1$, and with the following properties:

(i) The function $m_0$ defined as in \eqref{m0} is even, positive on $\R$, and extends to an analytic function $m_0:\mathcal{D}_{\vartheta}\to\C$ satisfying the identies
\begin{equation}\label{more10.1}
m_0(z)=m_0(-z)=\overline{m_0(\overline{z})}\qquad\text{ for any }z\in\mathcal{D}_{\vartheta}.
\end{equation}

(ii) Moreover, $m'_0(r)<0$ if $r\in(0,\infty)$ and
\begin{equation}\label{DefMaxg}
|m_0(z)|\lesssim (1+|z|)^{-d}\qquad\text{ for any }z\in\mathcal{D}_{\vartheta}.
\end{equation}

Given an acceptable equilibrium, we define its variance by the formula
\begin{equation}\label{a2def}
a_2:=\int_\R t^2m_0(t)\,dt.
\end{equation}
The equilibria we consider in this article are of two types:
\begin{enumerate}
\item ``thin-tail acceptable equilibria'': $d>3$, or $d=3$ and $a_2<\infty$;
\item ``fat-tail acceptable equilibria'': $1<d<3$, or $d=3$ and $a_2=\infty$.
\end{enumerate}
\end{definition}

Typical examples are polynomially decreasing or Gaussian equilibria, which lead to
\begin{equation}\label{m0_examples}
m_0(r)=\frac{c_d}{\left[1+r^2\right]^\frac{d}{2}},\qquad m_0(r)=\pi^{-1/2}e^{-r^2}.
\end{equation}
Notice that, using the Cauchy integral formula, we can upgrade \eqref{DefMaxg} to bounds on derivatives in a smaller region: for any $\delta>0$
\begin{equation}\label{DefMaxg2}
\sup_{z\in \mathcal{D}_{\vartheta-\delta}}(1+|z|)^{d+j}\vert \partial^j_zm_0(z)\vert\lesssim_{\delta,j} 1.
\end{equation}

Our first main result gives a precise description of Green's function associated to such acceptable equilibria.

\begin{theorem}\label{MainThm}

Assume that $d>1$, $M_0\in\mathcal{M}_{\vartheta,d}$ is an acceptable equilibrium in the sense of Definition \ref{DefinitionEquilibrium}, and assume that $r_0>0$ is sufficiently small. As before, we define
\begin{equation}\label{Kernels2}
\begin{split}
&K(\xi,\theta):=\int_{0}^\infty t\widehat{M}_0(t\xi)e^{-it\theta}dt,\\
&\widehat{G}(\xi,\tau):=\delta_0(\tau)-\mathfrak{1}_{[0,\infty)}(\tau)\frac{1}{2\pi}\int_{\R}\frac{K(\xi,\theta)}{1+K(\xi,\theta)}e^{i\theta\tau}\,d\theta.
\end{split}
\end{equation}
Then there exists $\gamma_0=\gamma_0(\vartheta,d,r_0)\in(0,\infty)$ such that the following claims hold:

(i) At high frequency, we have a perturbation of a kinetic density: if $|\xi|>r_0/2$, then we can write $\widehat{G}$ in the form
\begin{equation}\label{GFHF}
\widehat{G}(\xi,\tau)=\delta_0(\tau)+e^{-\gamma_0\tau\vert\xi\vert}\mathcal{E}_h(|\xi|,\tau),
\end{equation}
where, for any $a\in\{0,\ldots,10\}$, $b\in\{0,1\}$, $\tau\geq 0$, and $|r|>r_0/2$,
\begin{equation}\label{GFHF2}
\begin{split}
\tau^br^{a}\vert\partial_\tau^b\partial_r^{a}\mathcal{E}_h(r,\tau)\vert\lesssim r^{-1}.
\end{split}
\end{equation}

(ii) At low frequencies, we have an additional oscillatory component: if $|r|<2r_0$ then we can write $\widehat{G}$ in the form
\begin{equation}\label{GFLF}
\begin{split}
\widehat{G}(\xi,\tau)&=\delta_0(\tau)+\Re\big\{i[1+\mathfrak{m}_l(|\xi|)]e^{i\tau\omega(|\xi|)}\big\}+e^{-\gamma_0\tau\vert\xi\vert}\mathcal{E}_l(|\xi|,\tau)
\end{split}
\end{equation}
where, for any $a\in\{0,\ldots,10\}$, $b\in\{0,1\}$, $\tau\geq 0$, and $r<2r_0$,
\begin{equation}\label{GFLF2}
\begin{split}
r^{a}\vert\partial_r^{a}\mathfrak{m}_l(r)\vert\lesssim r^{d-1}+r^2\log(1/r),\qquad \tau^br^{a}\vert\partial_\tau^b\partial_r^{a}\mathcal{E}_l(r,\tau)\vert\lesssim r^{d-1}+r^2\log (1/r).
\end{split}
\end{equation}
The dispersion relation $\omega_1:=\Re\omega$ and the dissipation coefficient $\omega_2:=\Im\omega\ge0$ satisfy the bounds
\begin{equation}\label{GFLF3}
\begin{split}
&r^a|\partial_r^a(\omega_1(r)-1)|+r^a|\partial_r^a(\omega_2(r))|\lesssim r^{d-1}+r^2\log(1/r),\\
&\omega_2(r)\in\Big[\frac{-\pi m'_0(\omega_1(r)/r)}{4r^2},\frac{-\pi m'_0(\omega_1(r)/r)}{r^2}\Big],
\end{split}
\end{equation} 
for any $a\in\{0,\ldots,10\}$, $\tau\geq 0$, and $r<2r_0$.

In addition, in the thin-tail case ($d\geq 3$ and $a_2<\infty$), the functions $\mathfrak{m}_l$, $\mathcal{E}_l$, $\omega_1$, and $\omega_2$ satisfy the stronger bounds
\begin{equation}\label{GFLF4}
\begin{split}
r^{a}\vert\partial_r^{a}\mathfrak{m}_l(r)\vert\lesssim r^2,\qquad \tau^br^{a}\vert\partial_\tau^b\partial_r^{a}\mathcal{E}_l(r,\tau)\vert\lesssim r^{d-1}+r^3.
\end{split}
\end{equation}
and
\begin{equation}\label{GFLF5}
r^a|\partial_r^a(\omega_1(r)-1-3a_2r^2/2)|+r^a|\partial_r^a(\omega_2(r))|\lesssim r^{d-1}+r^4\log(1/r),
\end{equation} 
for any $a\in\{0,\ldots,10\}$, $b\in\{0,1\}$, $\tau\geq 0$, and $r<2r_0$, where $a_2$ is defined as in \eqref{a2def}.

If $d> 2n+1$ for some integer $n\ge 1$, then $\mathfrak{m}_l$ and $\omega\in C^n([0,\infty))$ and we have an expansion for $\omega_1$:
\begin{equation*}
r^a\vert\partial_r^a(\omega_1(r)-(1+b_2r^2+\dots+b_{2n}r^{2n})\vert+r^a\vert\partial_r^a(\omega_2(r))\vert\lesssim r^{d-1}+r^{2n+2}\log(1/r),
\end{equation*}
where $b_n$ only depends on moments of $m_0$ of order less than $n$. If $d=2n+3$ and the $2n+2$-th moment of $m_0$ is finite, one can remove the $\log$-loss.

\end{theorem}


We call the real part $\Re(\omega)$ of $\omega$ the \emph{dispersion relation}, and the imaginary part $\Im(\omega)$ the \emph{dissipation coefficient}. This is a departure from the classical physics literature which would call $\omega$ the dispersion relation; however, the real part $\Re(\omega)$ induces a dispersive dynamics akin to the dynamics of a Schr\"{o}dinger equation, while $\Im(\omega)$ induces a ``dissipation'' dynamics similar to the effect of the heat equation.

We conclude this subsection with a few remarks:
\begin{enumerate}[wide,labelwidth=!, labelindent=0pt,itemsep=1pt]
\item \emph{Analytic equilibria and contour analysis.}
Our proof makes use of the analyticity of the equilibria studied via complex analysis methods such as contour integration, an approach already pursued by Landau \cite{Lan1946} and by now standard in the physics literature \cite{BiTr,LLX1981,Pe} and more recent mathematical treatments \cite{BeMaMoLin,Deg,GlSh,HKNgRo,MouVil}. In particular, the recent papers \cite{BeMaMoLin,HKNgRo} have investigated certain classes of ``thin-tail'' equilibria, and give linear decay estimates for the associated density and electric field, for sufficiently nice initial particle distributions $f_0$. For acceptable equilibria, our analysis unifies, sharpens and extends these results; most importantly we provide sharp dissipation rates, which are expected to play a key role in the problem of nonlinear stability.\footnote{We note however that \cite{HKNgRo} also considers different classes of equilibria not covered by Theorem \ref{MainThm}, such as nonradial equilibria. We also note the very recent preprint \cite{Ngu2}, which appeared after this work was submitted, in which the author investigates several issues related to the linear stability of a class of radial equilibria.}

In the physically important setting of Maxwellians $m_0(r)=\pi^{-1/2}e^{-r^2}$ (studied in \cite{BeMaMoLin}), Theorem \ref{MainThm} confirms the precise, Gaussian nature of the dissipation coefficient, as can be seen directly from \eqref{GFLF3}, a folklore statement in the physics literature (see \cite[Section 5.3.6]{Bel2006} or \cite[Remark 8]{BeMaMoLin}), whose mathematical justification appears to be new.
\smallskip

\item \emph{Generalized Poisson equilibria and special cases.}
The case of $d=2$ in \eqref{m0_examples} corresponds to the Poisson kernel, for which a simple, explicit expression of Green's function can be found, as explained in Section \ref{RegularPoisson}. In this case, the dispersion relation is simply given as $\omega_1(\xi)=1$, and thus gives rise to a purely oscillatory dynamic. 

This case is part of a larger family arising as powers of the Poisson kernel (the ``generalized Poisson equilibria'') with polynomial decay of order $2j$, $j\in\mathbb{N}$, for which the formulas are still relatively explicit -- see Section \ref{GeneralPoisson}.
\smallskip
 
\item \emph{Dissipation and dispersion.}
At low frequencies, we see that the second term in \eqref{GFLF} exhibits both a damping effect (due to dissipation from $\Im\omega\ge0$) and a dispersive effect (due to the oscillatory phase $\Re\omega$ which has similar properties to Schr\"{o}dinger operators). The damping appears especially remarkable for a reversible system such as Vlasov-Poisson, and was already highlighted by Landau \cite{Lan1946,LLX1981}. 


The function $\omega$ is defined canonically by solving the equation
\begin{equation*}
\mathbf{k}(\omega(r)/r)=r^2,
\end{equation*}
by a fixed-point argument, where $r\ll 1$ and $\omega(r)=1+O(r)$. Here $\mathbf{k}$ is an analytic function that depends only on the equilibrium $m_0$. See Lemma \ref{LemContourSmallXi} for the precise construction. The function $\mathbf{k}$ can be expanded as a power series of $z^{-2}$ at infinity, 
\begin{equation*}
\mathbf{k}(z)=z^{-2}+3a_2z^{-4}+\ldots+ O(|z|^{-d-1}),
\end{equation*}
by continuing the expansion proved in Lemma \ref{LemG}, where the coefficients in the expansion only depend on moments of $m_0$ up to order $d-1$. Similarly, $\omega_2$ can be computed to any order by iterating the expansion in \eqref{more19.5}.
\smallskip

\item \emph{The role of decay of $M_0$.} 
Depending on the localization of the equilibrium $M_0$, the respective importance of dissipation and dispersion vary: for the slowly decaying Poisson equilibrium, we have a fast dissipation and no dispersion at all (only oscillations: $\omega(\xi)\equiv 1+i\vert\xi\vert$, see \eqref{GP}), while for compactly supported equilibria, there is no dissipation and one relies fully on the dispersion. Notice that faster decay of $m_0$ leads to slower dissipation. In particular, if $m_0$ is compactly supported, there is no dissipation at small frequency, a phenomenon that was already highlighted in \cite{GlSh,GlSh2}.
\end{enumerate}

\subsection{Various mechanisms of stability}\label{StabMechs}
In the linearized problem, as can be seen from \eqref{LinVolterra} one typically considers functions $h$ of the form
\begin{equation}\label{more44}
\begin{split}
h(x,t)=\int f_0(x-tv,v)dv,\qquad \Vert h(x,t)\Vert_{L^1}+\langle t\rangle^{3}\Vert h(x,t)\Vert_{L^\infty}\lesssim \Vert f_0\Vert_{L^1\cap L^\infty_x(L^\infty_v)}.
\end{split}
\end{equation}
Via the analysis of Green's function as in Theorem \ref{MainThm}, this leads to decay of $\rho$ and $E=\nabla\Delta^{-1}\rho$, and such decay properties play a vital role for a nonlinear stability analysis. Decay of $\rho$ is often referred to as ``Landau damping'' in the mathematics community, a term that seems to have evolved to cover any of the decay mechanisms involved -- dispersive, dissipative or kinetic. In order to provide some perspective towards the question of nonlinear stability, we give here a brief overview and comparison of some related settings and their associated decay mechanisms.

For this, let us examine the key formula \eqref{SolvingVolterra1} in combination with the information we have about Green's function according to Theorem \ref{MainThm}, and compare with the linearization around vacuum, where $\widehat{G}_{vac}(\xi,\tau)=\delta_0(\tau)$. 

For high frequencies, the setting of homogeneous equilibria (as studied in Section \ref{LinearAnalysis}) can be regarded as a perturbation of the vacuum case: from \eqref{GFHF} we see immediately that the inverse derivative in the electric field is inconsequential, and obtain that
 \begin{equation}\label{NicePropertiesEF}
 \begin{split}
 \Vert P_{high}E(t)\Vert_{L^1}+\langle t\rangle^3\Vert P_{high}E(t)\Vert_{L^\infty_x}\lesssim 1.
 \end{split}
 \end{equation}

However, at low frequencies, the situation is very different:
\begin{enumerate}
\item In the vacuum case, the above formula gives $\Vert E_{vac}\Vert_{L^\infty}=O(t^{-2})$ and $\Vert E_{vac}(t)\Vert_{L^2}=O(t^{-1/2})$. As a consequence, the nonlinear characteristics (which, as seen in \eqref{ReproducingChar}, roughly behave like $\int_{0}^t \tau E(\tau)d\tau$) exhibit a \emph{long range modification to scattering}, as was recently justified in \cite{FlOuPaWi,IoPaWaWiVacuum}.

\item In contrast, in the case of homogeneous equilibria, we can re-express the main terms in \eqref{GFLF} to get
\begin{equation}\label{NFG}
\chi_{\le r}(\xi)\widehat{G}(\xi,\tau)=\frac{d}{d\tau}\Re\left(\chi_{\le r}(\xi)e^{i\tau\omega}\mathfrak{1}_{\{\tau>0\}}\right)+l.o.t.
\end{equation}
Crucially, this allows to integrate by parts in $\tau$ in the solution formula \eqref{SolvingVolterra1} and shows that the main contribution in $\rho$ will be oscillatory. In particular, this implies that integral quantities of $E$ will be smaller than expected. As a consequence, in the nonlinear analysis one does \emph{not} expect long-range modifications to the characteristics, but rather \emph{linear scattering} for the particle distribution function. So far, this has been rigorously justified only in the simplest case of the Poisson equilibrium \cite{IoPaWaWiPoisson}.
\end{enumerate}

In practice, integration by parts in time (the method of \emph{normal forms}) leads to small denominators, and one needs much more subtle analysis to pass from linear to nonlinear results -- see Section \ref{NonlinStab} for an outline of this approach in the setting of the Poisson equilibrium. 

Finally we mention further settings which lie outside the scope of this article, but in which important progress on stable dynamics has been recently made:
\begin{enumerate}[wide,itemsep=1pt]

\item \emph{Screened interactions.} 
In some variations of the Vlasov-Poisson system associated to the ion dynamics, one considers \textit{screened} interactions. In this case, the low-frequencies are screened and a strong estimate like \eqref{NicePropertiesEF} holds globally. This gives favorable estimates and also leads to linear scattering \cite{BeMaMoScreened,HKNgRoScreened}, even in lower dimensions \cite{HuNgXu}.
\smallskip

\item \emph{The confined case $x\in\mathbb{T}^n$.}
The spatially periodic setting models a \emph{confined} plasma, and has been studied extensively. Under a suitable stability assumption on $M_0$, a so-called ``Penrose criterion'' ensuring that $1+K(\xi,\theta)\geq c>0$ holds uniformly, the linearizations around large classes of homogeneous equilibria can be treated as perturbations of free transport $\partial_t f+v\cdot\nabla_x f=0$. Exploiting the inherent phase mixing mechanism, one can show that regularity of $f_0$ leads to rapid decay of the density $\rho$, and ultimately again to linear scattering (see e.g.\ \cite{BMM2016,GrNgRo,MouVil}). In contrast to the setting on full space, here the dimension $n$ does not play a distinguished role. 

We highlight moreover that (as is well known \cite{BeMaMoLin,GlSh,GlSh2,HKNgRo}), such lower bounds on $1+K$ cannot hold on $\R^3$.
\smallskip

\item \emph{The case of localized equilibria in $\R^3$.}
In the context of galactic dynamics, spatially localized equilibria (with attractive interactions) play an important role. Classical physics conjectures and observations (see e.g.\ \cite{BiTr}) concern the dynamic behavior of perturbations of classes of such equilibria, which seem to exhibit oscillations. First steps towards a quantitative analysis of such linearized dynamics have been achieved in the recent works \cite{HRS2021,HRSS2023}.

At an extreme of spatial localization, the dynamics of Vlasov-Poisson near a (repulsive) point mass have been investigated recently in \cite{PW2020,PWY2022}, and reveal a decay rate as in the vacuum case, leading to a long range modification to the underlying characteristics.
\smallskip

\item \emph{Vlasov-HMF models.} One can also consider model problems, where the nonlinear interactions are simplified such as the {\it Vlasov-Hamiltonian Mean-Field} model \cite{BaBoDaRuYa}. In this case, one can often obtain simpler proofs of otherwise difficult results \cite{FaRo} and study (linearized) stability of inhomogeneous equilibria \cite{BaOlYa,FaHoRo}.

\end{enumerate}



\subsection{Nonlinear stability}\label{NonlinStab}
We consider now the full nonlinear asymptotic stability problem. In view of \eqref{Lan2}, this is given as
\begin{equation}\label{Lan4}
\rho(x,t)+\int_0^t\int_{\R^3}(t-s)\rho(x-(t-s)v,s)M_0(v)\,dvds=\mathcal{N}(x,t),
\end{equation}
for any $(x,t)\in\R^3\times[0,T]$, where
\begin{equation}\label{Lan5}
\begin{split}
\mathcal{N}(x,t)&:=\mathcal{N}_1(x,t)+\mathcal{N}_2(x,t),\\
\mathcal{N}_1(x,t)&:=\int_{\R^3}f_0(X(x,v,0,t),V(x,v,0,t))\,dv,\\
\mathcal{N}_2(x,t)&:=\int_0^t\int_{\R^3}\big\{E(x-(t-s)v,s)\cdot \nabla_vM_0(v)\\
&\qquad\qquad-E(X(x,v,s,t),s)\cdot (\nabla_vM_0)(V(x,v,s,t))\big\}\,dvds.
\end{split}
\end{equation}
The characteristic equations \eqref{Lan1} yield the reproducing formulas
\begin{equation}\label{ReproducingChar}
\begin{split}
X(x,v,s,t)&=x-(t-s)v+\int_s^t(\tau-s)E(X(x,v,\tau,t),\tau)\,d\tau,\\
V(x,v,s,t)&=v-\int_{s}^t E(X(x,v,\tau,t),\tau)\,d\tau.
\end{split}
\end{equation}
Since $E=\nabla_x\Delta_x^{-1}\rho$ can be recovered directly from the density $\rho$,  the equations \eqref{Lan4}--\eqref{ReproducingChar} yield a closed system for the density $\rho$ and the characteristic functions $X$ and $V$. The full distribution density $f$ can then be obtained from the formula \eqref{Lan2}.

To resolve the question of nonlinear stability of homogeneous equilibria, a precise understanding of the linear dynamics needs to be combined with methods that allow to control nonlinear interactions. In particular, the comparatively slow decay (essentially limited by the dimension of the ambient space $\R^3$) and the presence of small denominators and resonances is a serious challenge. 

The only result available to date concerns a particular ``fat-tail'' equilibrium falling in the framework of Theorem \ref{MainThm}, namely the Poisson equilibrium
\begin{equation}\label{NVP.2}
M_1(v):=\frac{C_1}{(1+|v|^2)^{2}},\qquad \widehat{M_1}(\xi)=e^{-|\xi|},
\end{equation}
for a suitable normalization constant $C_1>0$. Based on our analysis of the linearized dynamics near $M_1$ (see Sections \ref{LinearAnalysis}), we find two explicit, different dynamics in the associated electric field (see \eqref{eq:E-decomp}, or Section \ref{RegularPoisson} and Lemma \ref{SolvingVolterraP} below): an ``oscillatory'' component, which decays at almost the critical rate $\ip{t}^{-2}$ and oscillates as $e^{-it}$ in time, as well as a ``static'' part, which decays faster than the critical rate. 

For smooth, localized perturbations of $M_1$, this leads to global solutions that scatter to linear solutions:
 \begin{theorem}[\protect{\cite[Theorem 1.1]{IoPaWaWiPoisson}}]\label{thm:main_simple}
 There exists $\bar\varep>0$ such that if the initial particle distribution $f_0$ satisfies
 \begin{equation}\label{eq:init}
  \sum_{\abs{\alpha}+\abs{\beta}\leq 1}\norm{\ip{v}^{4.5}\partial_x^\alpha\partial_v^\beta f_0(x,v)}_{L^\infty_x L^\infty_v}+\norm{\ip{v}^{4.5}\partial_x^\alpha\partial_v^\beta f_0(x,v)}_{L^1_x L^\infty_v}\leq\varep_0\leq \bar\varep,
 \end{equation}
 then the Vlasov-Poisson system \eqref{VP} with $M_1$ defined as in \eqref{NVP.2} has a global unique solution $f\in C^1_{x,v,t}(\mathbb{R}^{3+3}\times\mathbb{R}_+)$ that scatters linearly, i.e.\ there exists $f_\infty\in L^\infty_{x,v}$ such that
 \begin{equation}\label{eq:lin_scatter}
\begin{split}
\Vert f(x,v,t)-f_\infty(x-tv,v)\Vert_{L^\infty_{x,v}}\lesssim\varepsilon_0\langle t\rangle^{-1/2}.
\end{split}
\end{equation}
Moreover, the electric field decomposes into a ``static'' and an ``oscillatory'' component with different decay rates
 \begin{equation}\label{eq:E-decomp}
\begin{split}
&E(t)=E^{stat}(t)+\Re(e^{-it}E^{osc}(t)),\\
&\ip{t}\norm{E^{stat}(t)}_{L^\infty}+\norm{E^{osc}(t)}_{L^\infty}\lesssim \varep_0\ip{t}^{-2+\delta},
\end{split}
 \end{equation}
where $\delta\in(0,1/100]$ is a small parameter.
\end{theorem}

This theorem was proved by the authors in \cite{IoPaWaWiPoisson} and appears to be the first nonlinear asymptotic stability result for the Vlasov-Poisson system in $\R^3$ in a neighborhood of a smooth non-trivial homogeneous equilibrium. We will review some of the main steps of the proof in Section \ref{outlineN}.

We close this section with some remarks for broader context.
\begin{enumerate}[wide,labelwidth=!, labelindent=0pt,itemsep=1pt]

\item \emph{Global existence.} Global existence for suitable perturbations of \eqref{VPO} is classical \cite{BD1985,Pfa1992,LP1991}, and extends to more complicated models \cite{WangRVP,WangVM}, but the precise asymptotic behavior of solutions is a delicate question which is our focus here. We refer to \cite{Be} for a survey of recent results, and we will discuss in more details the advances most relevant to our analysis.
\smallskip

\item \emph{Decay and stability.}
Working with the Poisson homogeneous equilibrium $M_1$ provides some convenient simplifications, as it leads to explicit formulas such as \eqref{GP}. However, as already discussed in Section \ref{StabMechs}, we expect that the conclusion of linear scattering as in the theorem and important aspects of the underlying analysis extend to more general smooth homogeneous acceptable equilibria, as in Definition \ref{DefinitionEquilibrium}.

The statement of the theorem, in particular the crucial decay estimates \eqref{eq:E-decomp}, depend on the fact that we work in dimension $n=3$. The decay is faster in higher dimensions $n\geq 4$ (making the result easier to obtain), but weaker and insufficient in dimension $n=2$. This is mainly due to dimension-dependent dispersive estimates like \eqref{more44}, and is in sharp contrast with the periodic case $x\in\T^n$.
\smallskip

\item \emph{Nonlinear stability results on $\R^3$.}
Except for Theorem \ref{thm:main_simple}, to the best of our knowledge, the nonlinear stability of the Vlasov-Poisson system is only understood near one of two scenarios: vacuum \cite{CK2016,FlOuPaWi,IoPaWaWiVacuum,Pan2020,Wang2018} and a repulsive point charge \cite{PW2020,PWY2022}. In both cases, the comparatively slow decay of the electric field leads to convergence along logarithmically modified linearized characteristics, a phenomenon known as \emph{modified scattering}. A similar effect has recently also been identified near vacuum in the Vlasov-Maxwell system \cite{Big2022}.
\smallskip

\item \emph{Nonlinear Landau damping on $\T^n$, and screened interactions.}
Nonlinear stability of Penrose stable homogeneous equilibria on $\T^n$ was proven in the pioneering work \cite{MouVil} (see also \cite{BMM2016,GrNgRo} for refinements and simplifications), after earlier works \cite{CM1998,HV2009}. While the linear analysis is comparatively simple in this setting, the crucial challenge for nonlinear stability are so-called \emph{nonlinear echoes}. Thanks to the fast decay due to phase mixing, these can be overcome for sufficiently smooth (Gevrey) perturbations. This allows for exponential decay of the electric field, and global solutions that scatter linearly. In this context, a high level of smoothness seems necessary, as some Sobolev perturbations can lead to other stationary solutions \cite{BGK1957,LZ2011}. Recently, such a nonlinear stability result has also been obtained for the ion equation \cite{GI2023}. We note that related mechanisms are involved in $2d$ \emph{fluid mixing}, e.g.\ near monotone shear flows \cite{BM15,Ionescu2018,IJ20,MZ20} or point vortices \cite{IJ2022}.

Such an analysis further extends to the screened case on $\R^3$ as investigated e.g.\ in \cite{BeMaMoScreened,HKNgRoScreened,HuNgXu}, where analogues of the Penrose condition still hold. Since there are no resonances in such settings, one can prove sufficiently rapid decay of the electric field for smooth perturbations.

\item \emph{Connection with Euler-Poisson models.}
In the related case of \emph{warm plasmas}, one uses two-fluid instead of kinetic models. The counterpart to \eqref{VPO}, corresponding to the electron dynamics, has similar dispersion relation as $\omega_1$ in \eqref{GFLF5} (the so-called ``Langmuir or Bohm-Gross'' waves), and was shown to be asymptotically stable (for irrotational data) in \cite{Guo1998,GeMaPa}. The work \cite{BeMaMoLin} emphasized a link between the kinetic and fluid model in the linearized equation whose precise  understanding and relevance for the nonlinear problem is an exciting challenge.

The counterpart to Vlasov-Poisson with screened interactions, the ion equation, is more involved and was obtained more recently in \cite{GP2011}, while the full two-fluid models involving both electrons and ions lead to new and interesting phenomena \cite{GIP2016,GM2014,IoPau1} whose kinetic counterpart is not yet well-understood.
\end{enumerate}

%
%

\subsection{Organization} The rest of this paper is organized as follows. In Section 2 we discuss the linear theory: we provide a complete proof of Theorem \ref{MainThm} and calculate explicitly the Green's function associated to a family of acceptable equilibria that generalize the Poisson equilibrium defined in \eqref{NVP.2}. In Section 3 we outline the main ideas in the proof of Theorem \ref{thm:main_simple}, following our recent work \cite{IoPaWaWiPoisson}.

\section{Linear analysis: proof of Theorem \ref{MainThm}}

In this section we prove Theorem \ref{MainThm}. We discuss first some explicit examples, the generalized Poisson equilibria, and then prove the theorem in the general case.

\subsection{The generalized Poisson equilibria}\label{GeneralPoisson}

We consider the generalized Poisson equilibria
\begin{equation*}
M_{j}(y):=\frac{c_{j}}{\left[1+\vert y\vert^2\right]^{1+j}},\qquad m_{j}(r):=\frac{c^\prime_{j}}{(1+r^2)^j}, \qquad j\in\{1,2,\ldots\},
\end{equation*}
where the constants $c_{j}$, $c^\prime_{j}$ are chosen such that $$\int_{\R^3} M_{j}(y)\,dy=\int_\R m_j(r)\,dr=1.$$ We calculate first the associated kernels $K_{j}$ defined as in \eqref{Kernels2}.

\begin{lemma}\label{ExplicitFormKPjLem}

Let $z:=\theta-i\vert\xi\vert$. Then
\begin{equation}\label{ExplicitKj}
\begin{split}
K_{j}(\xi,\theta)=-z^{-2}\sum_{p=0}^{j-1} a^{(j)}_p\left(\frac{\vert\xi\vert}{iz}\right)^p=-\frac{1}{z^2}+\frac{2i\vert\xi\vert }{z^3}+\dots
\end{split}
\end{equation}
where the coefficients $\{a^{(j)}_p\}_{0\le p\le j-1}$ are positive numbers defined recursively by
\begin{equation}\label{coeffi2}
\begin{split}
&a^{(1)}_0:=1,\qquad N^{(j+1)}:=\sum_{k=0}^{j-1}\frac{a^{(j)}_k}{2^k(k+1)},\\
&a^{(j+1)}_p:=\begin{cases}
\frac{(p+1)2^{p-1}}{N^{(j+1)}}\sum_{k=p-1}^{j-1}\frac{1}{2^{k}(k+1)}a^{(j)}_k,\qquad 1\le p\le j,\\
1,\qquad p=0.
\end{cases}
\end{split}
\end{equation}
\end{lemma}

In particular, we can compute the first few kernels:
\begin{equation}\label{ExplicitKGP}
\begin{split}
K_{1}(\xi,\theta)&=-\frac{1}{z^2},\quad 
K_{2}(\xi,\theta)=-\frac{1}{z^2}+\frac{2i\vert\xi\vert}{z^3},\quad
K_{3}(\xi,\theta)=-\frac{1}{z^2}+\frac{2i\vert\xi\vert}{z^3}+\frac{2\vert\xi\vert^2}{z^4}.
\end{split}
\end{equation}

\begin{proof}[Proof of Lemma \ref{ExplicitFormKPjLem}]

We start by computing $\widehat{m_j}(s)$. When $j=1$, we have the usual Poisson kernel $\widehat{m_1}(s)=e^{-\vert s\vert}$. For $j\ge2$, the formula is obtained (up to a normalizing multiplicative constant) by iterated convolution: $\widehat{m_{j+1}}=c\cdot \widehat{m_j}\ast\widehat{m_1}$. Letting $\ast^j\widehat{m_1}(r)=Q_j(\vert r\vert)e^{-\vert r\vert}$ where $Q_j$ is a polynomial of order $j-1$, we find that, for $x\ge0$,
\begin{equation*}
\begin{split}
Q_{j+1}(x)e^{-x}&=e^{-x}\int_{-\infty}^0Q_j(-y)e^{2y}dy+e^{-x}\int_{0}^xQ_j(y)dy+e^x\int_{x}^\infty Q_j(y)e^{-2y}dy\\
&=e^{-x}\left\{I_-[Q_j]+I_0[Q_j]+I_+[Q_j]\right\}
\end{split}
\end{equation*}
and the action of the integral operators can easily be computed on a basis element. Indeed, we notice that $I_-[x^n]=(n/2)I_-[x^{n-1}]$ when $n\ge 1$ and $I_{\pm}[1]=1/2$, while $I_+[x^{n}]=x^n/2+(n/2)I_+[x^{n-1}]$, $n\ge 1$. Thus
\begin{equation*}
\begin{split}
I_-[x^n]=\frac{n!}{2^{n+1}},\qquad I_0[x^n]=\frac{x^{n+1}}{n+1},\qquad I_+[x^n]=\sum_{k=0}^n\frac{n!}{k!2^{n+1-k}}x^k,
\end{split}
\end{equation*}
and we obtain the following recursive formula: if $Q_j=\sum_{p=0}^{j-1}d_p^{(j)}x^p$, then
\begin{equation*}
\begin{split}
Q_{j+1}&=\sum_{p=0}^{j}d_p^{(j+1)}x^p,\qquad d_p^{(j+1)}:=\frac{2^p}{p!}\sum_{k=p-1}^{j-1}\frac{k!}{2^{k+1}}d_k^{(j)},\quad p\ge 1,\qquad  d_0^{(j+1)}:=\sum_{k=0}^{j-1}\frac{k!}{2^k}d^{(j)}_k.
\end{split}
\end{equation*}
In particular, we observe that
\begin{equation*}
\begin{split}
Q_{j+1}(0)=Q_{j+1}^\prime(0)=\sum_{k=0}^{j-1}\frac{k!}{2^k}d^{(j)}_k.
\end{split}
\end{equation*}

Now, using the formula
\begin{equation*}
\begin{split}
\int_{0}^\infty t(t\vert\xi\vert)^pe^{-t\vert\xi\vert}e^{-it\theta}dt&=\vert\xi\vert^p\int_{0}^\infty t^{p+1}e^{-itz}dt=\vert\xi\vert^p\frac{(p+1)!}{(iz)^{p+2}}
\end{split}
\end{equation*}
we see that
\begin{equation*}
\begin{split}
\int_{t=0}^\infty tQ_j(t\vert\xi\vert)e^{-t\vert\xi\vert}e^{-it\theta}dt=-z^{-2}\sum_{p=0}^{j-1} (p+1)!d^{(j)}_p\left(\frac{\vert\xi\vert}{iz}\right)^p.
\end{split}
\end{equation*}
The conclusion of the lemma follows since $\widehat{m_j}(r)=Q_j(|r|)e^{-|r|}/Q_j(0)$. 
\end{proof}

\subsubsection{The regular Poisson kernel}\label{RegularPoisson}

We can see from \eqref{Kernels2} that the important function is the Fourier transform of $K(1+K)^{-1}$. This can easily be computed from Lemma \ref{ExplicitFormKPjLem} and the Fourier formula (coming from the Cauchy integral whenever $\Im(\zeta)>0$ and $\tau\geq 0$)
\begin{equation}\label{GenForm}
\int_{-\infty}^\infty\frac{e^{i\tau\theta}d\theta}{\theta-\zeta}=2\pi i e^{i\tau\zeta}.
\end{equation}
In the case of the regular Poisson equilibrium, we find that
\begin{equation*}
\begin{split}
\frac{K_1(\xi,\theta)}{1+K_1(\xi,\theta)}=-\frac{1}{2}\left(\frac{1}{\theta-(1+i\vert\xi\vert)}-\frac{1}{\theta-(-1+i\vert\xi\vert)}\right),
\end{split}
\end{equation*}
and, using the definition \eqref{Kernels2},
\begin{equation}\label{GP}
\widehat{G_1}(\xi,\tau)=\delta_0(\tau)- \mathfrak{1}_{[0,\infty)}(\tau)e^{-\vert\xi\vert\tau}\sin(\tau).
\end{equation}
In particular, we notice that
\begin{equation}\label{GPex}
\widehat{G_1}(\xi,\tau)=e^{-\tau\vert\xi\vert}\frac{d}{d\tau}\left(\cos(\tau)\mathfrak{1}_{[0,\infty}(\tau)\right).
\end{equation}
This second formula is important in nonlinear analysis, as explained in \cite{IoPaWaWiPoisson}, as it allows us to integrate by parts in $\tau$ in the integral \eqref{SolvingVolterra1} (after suitable localizations), to obtain a second identity for the density $\rho$. This integration by parts in $\tau$ is analogous to the method of normal forms. See Lemma \ref{SolvingVolterraP} for the precise identities.

\subsubsection{The first generalized Poisson kernel} In view of \eqref{ExplicitKGP} we have
\begin{equation*}
\begin{split}
\frac{K_2(\xi,\theta)}{1+K_2(\xi,\theta)}&=\frac{\zeta+2\vert\xi\vert}{\zeta^3+\zeta+2\vert\xi\vert},\qquad\zeta=iz=\vert\xi\vert+i\theta.
\end{split}
\end{equation*}
Now, we see easily that the polynomial $Q_2(\zeta)=\zeta^3+\zeta+2\vert\xi\vert$ has one real root $r_0<0$ and two complex conjugate roots $r_1=\rho+i\kappa=\overline{r_2}$, $\kappa>0$. Identifying coefficients with symmetric functions of the roots, we find that
\begin{equation}\label{SymCoefP2}
\begin{split}
r_0=-2\rho<0,\qquad 1+3\rho^2=\kappa^2,\qquad \vert\xi\vert=\rho+4\rho^3.
\end{split}
\end{equation}
We can expand in partial fractions
\begin{equation*}
 \frac{\zeta+2\vert\xi\vert}{\zeta^3+\zeta+2\vert\xi\vert}=\frac{-2\alpha}{\zeta+2\rho}+\frac{\alpha+i\beta}{\zeta-(\rho+i\kappa)}+\frac{\alpha-i\beta}{\zeta-(\rho-i\kappa)},
\end{equation*}
where
\begin{equation*}
 1=4\rho\alpha+2\rho\alpha-2\kappa\beta,\qquad 2\vert\xi\vert=-2\alpha(\rho^2+\kappa^2)-2\rho(2\rho\alpha+2\kappa\beta).
\end{equation*}
Since $\kappa^2=1+3\rho^2$ and $|\xi|=\rho+4\rho^3$ we find
\begin{equation}\label{PFDP3}
 \alpha=-\frac{4\rho^3}{1+12\rho^2},\qquad \beta=-\frac{1}{2\kappa}\Big(1+\frac{24\rho^4}{1+12\rho^2}\Big),
\end{equation}
and we arrive at the formula
\begin{equation*}
\begin{split}
\frac{K_2(\xi,\theta)}{1+K_2(\xi,\theta)}&=\frac{2i\alpha}{\theta-i(\vert\xi\vert-r_0)}+\frac{\beta-i\alpha}{\theta-i(\vert\xi\vert-r_1)}+\frac{-\beta-i\alpha}{\theta-i(\vert\xi\vert-\overline{r_1})}.
\end{split}
\end{equation*}
Using \eqref{Kernels2} and the Cauchy integral formula \eqref{GenForm}, this gives
\begin{equation*}
\begin{split}
\widehat{G_2}(\xi,\tau)&=\delta_0(\tau)+\mathfrak{1}_{[0,\infty)}(\tau)\big[2 \alpha e^{-(2\rho+\vert\xi\vert)\tau}-2e^{-(\vert\xi\vert-\rho)\tau}\Re\left\{(\alpha+i\beta)e^{i\kappa\tau}\right\}\big].
\end{split}
\end{equation*}
Using \eqref{SymCoefP2}--\eqref{PFDP3} we can expand when $|\xi|\ll 1$
\begin{equation}\label{expand3}
\begin{split}
&\rho=|\xi| +O(|\xi|^3),\qquad \kappa=1+\frac{3|\xi|^2}{2}+O(|\xi|^4),\\
&\alpha=-4\vert\xi\vert^3+O(\vert\xi\vert^5),\qquad\beta=-\frac{1}{2}+\frac{3}{4}\vert\xi\vert^2+O(\vert\xi\vert^4).
\end{split}
\end{equation}

\subsubsection{The second generalized Poisson kernel}

For the second generalized Poisson kernel, we have, with $z=\theta-i\vert\xi\vert$ and $\zeta=iz=|\xi|+i\theta$ as before,
\begin{equation*}
\begin{split}
K_{3}(\xi,\theta)&=-\frac{1}{z^2}+\frac{2i\vert\xi\vert}{z^3}+\frac{2\vert\xi\vert^2}{z^4},\\
\frac{K_3(\xi,\theta)}{1+K_3(\xi,\theta)}&=\frac{\zeta^2+2\vert\xi\vert\zeta+2\vert\xi\vert^2}{\zeta^4+\zeta^2+2\vert\xi\vert\zeta+2\vert\xi\vert^2}.
\end{split}
\end{equation*}
The polynomial $Q_3(\zeta):=\zeta^4+\zeta^2+2\vert\xi\vert\zeta+2\vert\xi\vert^2=\zeta^4+(\zeta+\vert\xi\vert)^2+\vert\xi\vert^2$ has two couples of complex conjugate roots $r_1=\overline{r_2}=\rho_1+i\kappa_1$, $r_3=\overline{r_4}=\rho_3+i\kappa_3$, $\kappa_1,\kappa_3>0$. In addition, we have that
\begin{equation*}
\begin{split}
\rho_1+\rho_3=0,&\quad \rho_1^2+\kappa_1^2+4\rho_1\rho_3+\rho_3^2+\kappa_3^2=1,\\
 2\rho_1(\rho_3^2+\kappa_3^2)+2\rho_3(\rho_1^2+\kappa_1^2)=-2\vert\xi\vert, &\quad (\rho_1^2+\kappa_1^2)(\rho_3^2+\kappa_3^2)=2\vert\xi\vert^2,
\end{split}
\end{equation*}
from which we deduce that
\begin{equation*}
\begin{split}
0<\rho:=\rho_1=-\rho_3,&\qquad \kappa_1^2+\kappa_3^2-2\rho^2=1,\\
\rho(\kappa_3^2-\kappa_1^2)=-\vert\xi\vert&\qquad \rho^4+\rho^2(\kappa_1^2+\kappa_3^2)+\kappa_1^2\kappa_3^2=2\vert\xi\vert^2.
\end{split}
\end{equation*}
Therefore
\begin{equation}\label{rootsId}
\begin{split}
&2\kappa_1^2=1+2\rho^2+\frac{\vert\xi\vert}{\rho}\qquad 2\kappa_3^2=1+2\rho^2-\frac{\vert\xi\vert}{\rho},\\
&16\rho^6+8\rho^4+\rho^2(1-8\vert\xi\vert^2)=\vert\xi\vert^2.
 \end{split}
\end{equation}

We now have the partial fraction decomposition
\begin{equation*}
\begin{split}
\frac{\zeta^2+2\vert\xi\vert \zeta+2\vert\xi\vert^2}{\zeta^4+\zeta^2+2\vert\xi\vert \zeta+2\vert\xi\vert^2}&=\frac{a+ib}{\zeta-r_1}+\frac{a-ib}{\zeta-\overline{r_1}}+\frac{c+id}{\zeta-r_3}+\frac{c-id}{\zeta-\overline{r_3}}\\
&=\frac{2a\zeta-2(a\rho+b\kappa_1)}{(\zeta-\rho)^2+\kappa_1^2}+\frac{2c\zeta-2(-c\rho+d\kappa_3)}{(\zeta+\rho)^2+\kappa_3^2},
\end{split}
\end{equation*}
and bringing to the same denominator gives 
\begin{equation*}
\begin{split}
a+c=0,\qquad (a-c)\rho-(b\kappa_1+d\kappa_3)&=\frac12,\\
a(-\rho^2+\kappa_3^2)+c(-\rho^2+\kappa_1^2)+2\rho(-b\kappa_1+d\kappa_3)&=\vert\xi\vert,\qquad \\
(a\rho+b\kappa_1)(\rho^2+\kappa_3^2)+(-c\rho+d\kappa_3)(\rho^2+\kappa_1^2)&=-\vert\xi\vert^2.
\end{split}
\end{equation*}
After algebraic simplifications and using also \eqref{rootsId}, we have
\begin{equation*}
\begin{split}
a+c=0,\qquad b\kappa_1+d\kappa_3=2a\rho-\frac{1}{2},\qquad -a\vert\xi\vert+2\rho^2(-b\kappa_1+d\kappa_3)=\vert\xi\vert\rho\\
a\rho(1+4\rho^2)+b\kappa_1\Big[2\rho^2+\frac{\rho-\vert\xi\vert}{2\rho}\Big]+d\kappa_3\Big[2\rho^2+\frac{\rho+\vert\xi\vert}{2\rho}\Big]=-\vert\xi\vert^2.
\end{split}
\end{equation*}
Using the second and third identities, the last identity gives
\begin{equation*}
a\rho(1+4\rho^2)+\frac{(4\rho^2+1)(4a\rho-1)}{4}+\frac{(a|\xi|+\rho|\xi|)|\xi|}{4\rho^3}=-\vert\xi\vert^2.
\end{equation*}
Therefore
\begin{equation}\label{rootsId2}
\begin{split}
a&=\rho(1+4\rho^2)\frac{\rho^2-\vert\xi\vert^2}{8\rho^4(1+4\rho^2)+\vert\xi\vert^2},\qquad c=-a,\\
b\kappa_1&=-\frac{1}{4}+a\rho-\frac{|\xi|}{4\rho}-\frac{a\vert\xi\vert}{4\rho^2},\qquad d\kappa_3=-\frac{1}{4}+a\rho+\frac{|\xi|}{4\rho}+\frac{a\vert\xi\vert}{4\rho^2}.
\end{split}
\end{equation}

Using also the formula \eqref{GenForm}, we deduce that
\begin{equation*}
\begin{split}
\frac{K_3(\xi,\tau)}{1+K_3(\xi,\theta)}&=\frac{b-ia}{\theta-(\kappa_1+i(\vert\xi\vert-\rho))}-\frac{b+ia}{\theta-(-\kappa_1+i(\vert\xi\vert-\rho))}\\
&+\frac{d+ia}{\theta-(\kappa_3+i(\vert\xi\vert+\rho))}-\frac{d-ia}{\theta-(-\kappa_3+i(\vert\xi\vert+\rho))},\\
\widehat{G_3}(\xi,\tau)=\delta_0(\tau)&-\mathfrak{1}_{[0,\infty)}(\tau)\big\{2e^{-(\vert\xi\vert-\rho)\tau}\Re\left\{(a+ib)e^{i\kappa_1\tau}\right\}+2e^{-(\rho+\vert\xi\vert)\tau}\Re\{(-a+id)e^{i\kappa_3\tau}\}\big\}.
\end{split}
\end{equation*}
Finally, using \eqref{rootsId}--\eqref{rootsId2} we can expand when $|\xi|\ll 1$
\begin{equation}\label{expand6}
\begin{split}
&\rho=|\xi|-8|\xi|^5+O(|\xi|^7),\qquad \kappa_1=1+\frac{|\xi|^2}{2}+O(|\xi|^4),\qquad\kappa_3=|\xi|-2|\xi|^3+O(|\xi|^5),\\
&a=-16|\xi|^5+O(|\xi|^7),\qquad b=-\frac{1}{2}+O(|\xi|^2),\qquad d=-2|\xi|^3+O(|\xi|^5).
\end{split}
\end{equation}

\begin{figure}[h]
 \centering
 \includegraphics[width=0.6\textwidth]{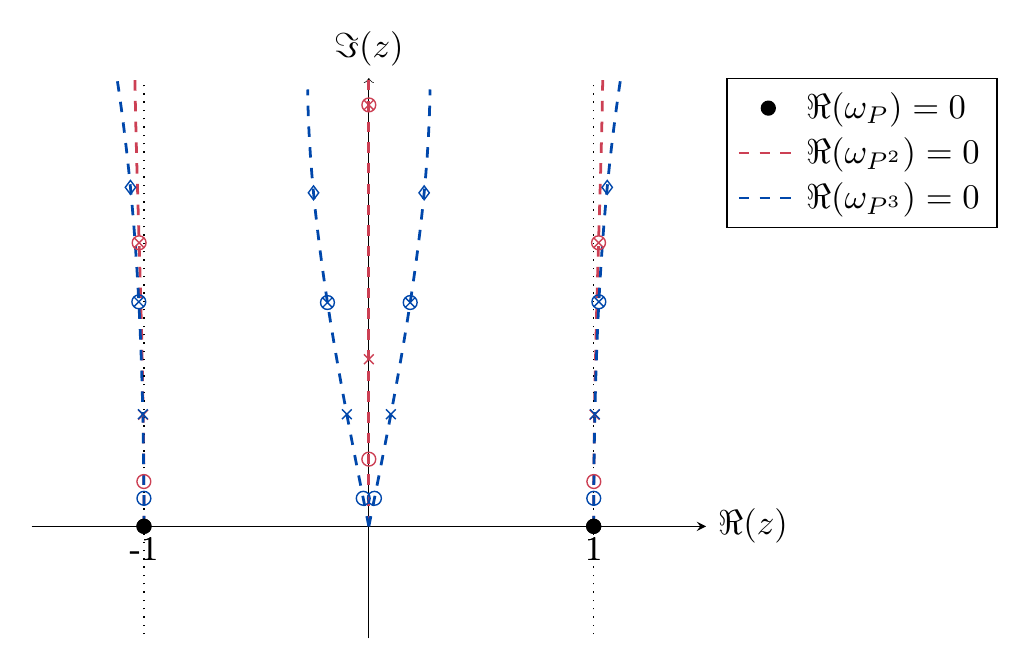}
 \caption{The curves of poles for the first two generalized Poisson equilibria. As $|\xi|$ increases, the poles move upwards, as indicated by the marks along the curves for sample values of $|\xi|$.}
 \label{fig:poles}
\end{figure}

\subsubsection{Higher order generalized Poisson kernels}

In general, it is more difficult to isolate the poles of $K_j(1+K_j)^{-1}$. However, using Lemma \ref{ExplicitFormKPjLem} we see that
\begin{equation*}
\frac{K_j(\xi,\theta)}{1+K_j(\xi,\theta)}=\frac{\sum_{p=0}^{j-1} a^{(j)}_p|\xi|^p(iz)^{j-1-p}}{(iz)^{j+1}+\sum_{p=0}^{j-1} a^{(j)}_p|\xi|^p(iz)^{j-1-p}}=\frac{\sum_{p=0}^{j-1} a^{(j)}_p|\xi|^p\zeta^{j-1-p}}{\zeta^{j+1}+\sum_{p=0}^{j-1} a^{(j)}_p|\xi|^p\zeta^{j-1-p}},
\end{equation*}
where $\zeta=iz=|\xi|+i\theta$. As before, let
\begin{equation*}
Q_j(\zeta):=\zeta^{j+1}+\sum_{p=0}^{j-1} a^{(j)}_p|\xi|^p\zeta^{j-1-p}=\zeta^{j+1}+\zeta^{j-1}+2|\xi|\zeta^{j-2}+\ldots.
\end{equation*}
It follows from Rouch\'e's theorem that if $\vert\xi\vert\ll 1$ is small enough, the polynomial $Q_j$ has one root around $\zeta=\pm i$ and $j-1$ roots in the ball $B(0,C_j\vert\xi\vert)$, for some constant $C_j$. 

We can also prove that all these roots have real part strictly smaller than $|\xi|$, which allows us to apply formula \eqref{GenForm} to prove exponential decay (this follows from the fact that $a^{(j)}_p\ge 0$; we give a more general proof in Proposition \ref{PropG1}).
In addition, we can use the implicit function theorem to study the branch of the root bifurcating from $i$ and get that
\begin{equation}\label{Pole1ForGeneralizedPoisson}
\begin{split}
\zeta_1=|\xi|+i\Big(1+\frac{3-a^{(j)}_2}{2}\vert\xi\vert^2\Big)+O(\vert\xi\vert^3),
\end{split}
\end{equation}
when $|\xi|\ll 1$. We note that by construction $0\leq a_2^{(j)}<3$.

\subsection{The general case} We consider now the case of general acceptable equilibria, as in Definition \ref{DefinitionEquilibrium}, when the linearized equation cannot be solved as explicitly. From \eqref{Kernels2}, we see that we need to study the kernel $K$, and using the notations above, we can express $K$ in terms of a single holomorphic function ${\bf k}$:
\begin{equation}\label{DefK}
\begin{split}
K(\xi,\theta)&=\int_{0}^\infty t\widehat{m_0}(t\vert \xi\vert)e^{-it\theta}dt=-\frac{i}{\vert\xi\vert^2}\int_{0}^\infty \widehat{m_0^\prime}(s)e^{-is\theta/\vert\xi\vert}ds=-\frac{1}{\vert\xi\vert^2}\mathbf{k}(\frac{\theta}{\vert\xi\vert}).
\end{split}
\end{equation}

\begin{remark}
In the case of the generalized Poisson equilibria, it follows from \eqref{ExplicitKGP} that
\begin{equation}\label{ExplicitkGP}
\begin{split}
{\bf k}_1(z)&=-(1+iz)^{-2},\qquad {\bf k}_{2}(z)=-(1+iz)^{-2}-2(1+iz)^{-3},\\
{\bf k}_{3}(z)&=-(1+iz)^{-2}-2(1+iz)^{-3}-2(1+iz)^{-4}.
\end{split}
\end{equation}
\end{remark}

It is apparent from the formula that $\mathbf{k}$ extends to a holomorphic function on the lower half-plane $\mathbb{H}_-:=\{z\in\C:\,\Im z< 0\}$, which can easily be calculated using Fubini:
\begin{equation}\label{DefG}
\begin{split}
\mathbf{k}(z)&=i\int_{0}^\infty\left(\int_{-\infty}^\infty m_0^\prime(t)e^{-is(t+z)}dt\right)ds=\int_{-\infty}^\infty\frac{m_0^\prime(t)}{t+z}dt=-\int_{-\infty}^\infty\frac{m_0^\prime(t)}{z-t}dt.
\end{split}
\end{equation}
This formula extends to the real axis using the Plemelj formula: for $z=x+iy$,
\begin{equation*}
\begin{split}
\frac{1}{z-t}=\frac{x-t}{(x-t)^2+y^2}-i\frac{y}{\vert y\vert}\frac{\vert y\vert}{(x-t)^2+y^2}.
\end{split}
\end{equation*}
Therefore
\begin{equation}\label{kRealAxis}
\begin{split}
\lim_{y\to 0,\,y<0}\mathbf{k}(x+iy)&=-\hbox{p.v.}\int_{-\infty}^\infty \frac{m'_0(t)}{x-t}dt-i\pi m_0^\prime(x)=-\pi\mathcal{H}(m'_0)(x)-i\pi m_0^\prime(x),
\end{split}
\end{equation}
where $\mathcal{H}$ denotes the Hilbert transform on $\R$. We start with a general useful criterion from Penrose \cite{Pe}. 

\begin{proposition}[Penrose]\label{PropG1}

Assume that $m_0:\R\to\R$ is an even $C^3$ function such that $m_0$ is decreasing on $[0,\infty)$ and $m'_0(t)(1+t^2)\in L^\infty(\R)$. We define the function $\mathbf{k}:\overline{\mathbb{H}_-}\to\C$ as in \eqref{DefG}--\eqref{kRealAxis}.

Then $\mathbf{k}$ is holomorphic on $\mathbb{H}_-$ and continuous on $\overline{\mathbb{H}_-}\cup\{\infty\}$ with $\mathbf{k}(\infty)=0$. In addition, $\mathbf{k}(0)\in(-\infty,0)$ and we have an expansion around $0$,
\begin{equation}\label{ExpG0}
\begin{split}
\mathbf{k}(z)=\mathbf{k}(0)-i\pi m_0''(0)z+O(|z|^2).
\end{split}
\end{equation}
Moreover, $\mathbf{k}(\overline{\mathbb{H}_-})\subseteq\mathbb{C}\setminus [0,\infty)$ and $\mathbf{k}$ satisfies the functional equation
\begin{equation}\label{FunctionalEq1}
\mathbf{k}(-\overline{z})=\overline{\mathbf{k}(z)}.
\end{equation}
In particular, if $\xi\ne 0$ and $\theta\in\overline{\mathbb{H}_-}$ then
\begin{equation}\label{NonVan}
1+K(\xi,\theta)\neq 0.
\end{equation}
\end{proposition}

\begin{proof}[Proof of Proposition \ref{PropG1}]

Direct inspection of \eqref{DefG} leads to \eqref{FunctionalEq1}. In addition, since $m_0^\prime(0)=0$, the imaginary part of $\mathbf{k}(0)$ vanishes and we find that
\begin{equation}\label{kof0}
\begin{split}
\mathbf{k}(0)=-\lim_{\varepsilon\to 0}\int_{|t|\geq\varepsilon} \frac{m'_0(t)}{-t}dt=\int_{\R} \frac{m'_0(t)}{t}\,dt<0,
\end{split}
\end{equation}
since $m'_0$ is odd and negative on $[0,\infty)$. Moreover, if $z=x+iy$, $y<0$, and $m'_0(t)=At\mathfrak{1}_{[-1,1]}(t)+t^2B(t)$, $A=m''_0(0)$, with $B$ odd satisfying $|B'(t)|\lesssim 1$ if $|t|\leq 1$, then 
\begin{equation}\label{kof0exp}
\begin{split}
\big|\mathbf{k}(z)-\mathbf{k}(0)&+i\pi m_0''(0)z\big|=\Big|\int_{\R}\frac{m_0^\prime(t)}{t-z}\,dt-\int_{\R} \frac{m'_0(t)}{t}\,dt+i\pi m_0''(0)z\Big|\\
&=|z|\Big|\int_{\R}\frac{m_0^\prime(t)}{(t-z)t}\,dt+i\pi m_0''(0)\Big|\\
&\leq|z|\Big|\int_{-1}^1\frac{A}{t-x-iy}\,dt+i\pi A\Big|+|z|\Big|\int_{\R}\frac{t^2B(t)}{(t-z)t}-\frac{t^2B(t)}{t^2}\,dt\Big|\\
&\lesssim|z|^2.
\end{split}
\end{equation}
This proves \eqref{ExpG0}.

The definition shows easily that $\lim_{|z|\to\infty,\,z\in\overline{\mathbb{H}_-}}\mathbf{k}(z)=0$. Let $C$ denote the image of $\partial\mathbb{H}_-$ under $\mathbf{k}$ (see Figure \ref{fig:scatter}). This is a continuous curve, $C^1$ except at $\mathbf{k}(\infty)=0$. We see from \eqref{kRealAxis} that the only other place where the curve crosses the real axis is at the minimum, $z=0$, where $\mathbf{k}(0)< 0$. Let $x>0$ be any positive real number. From the considerations above, the index of $x$ with respect to $C$ is $0$, thus
\begin{equation*}
\begin{split}
0=\hbox{Ind}_{x}(C)=\frac{1}{2\pi i}\int_{C}\frac{dz}{z-x}=\frac{1}{2\pi i}\int_{\partial\mathbb{H}_-}\frac{\mathbf{k}^\prime(\zeta)}{\mathbf{k}(\zeta)-x}d\zeta.
\end{split}
\end{equation*}
On the other hand, by contour deformation (the so-called ``argument principle'') it is easy to see that if $f$ is meromorphic inside a domain $D$ bounded by a curve $C$, then
\begin{equation*}
\begin{split}
\frac{1}{2\pi i}\int_C\frac{f^\prime(z)}{f(z)}dz=-Z+P=-\#\{f=0\}+\#\{f=\infty\}.
\end{split}
\end{equation*}
Applying this  on $C$ bounding $\mathbb{H}_-$ to $f(z)=\mathbf{k}(z)-x$, which has no pole in $\mathbb{H}_-$, we see that it also never vanishes in $\mathbb{H}_-$.
\end{proof}

\begin{figure}[h]
 \centering
 \includegraphics[width=\textwidth]{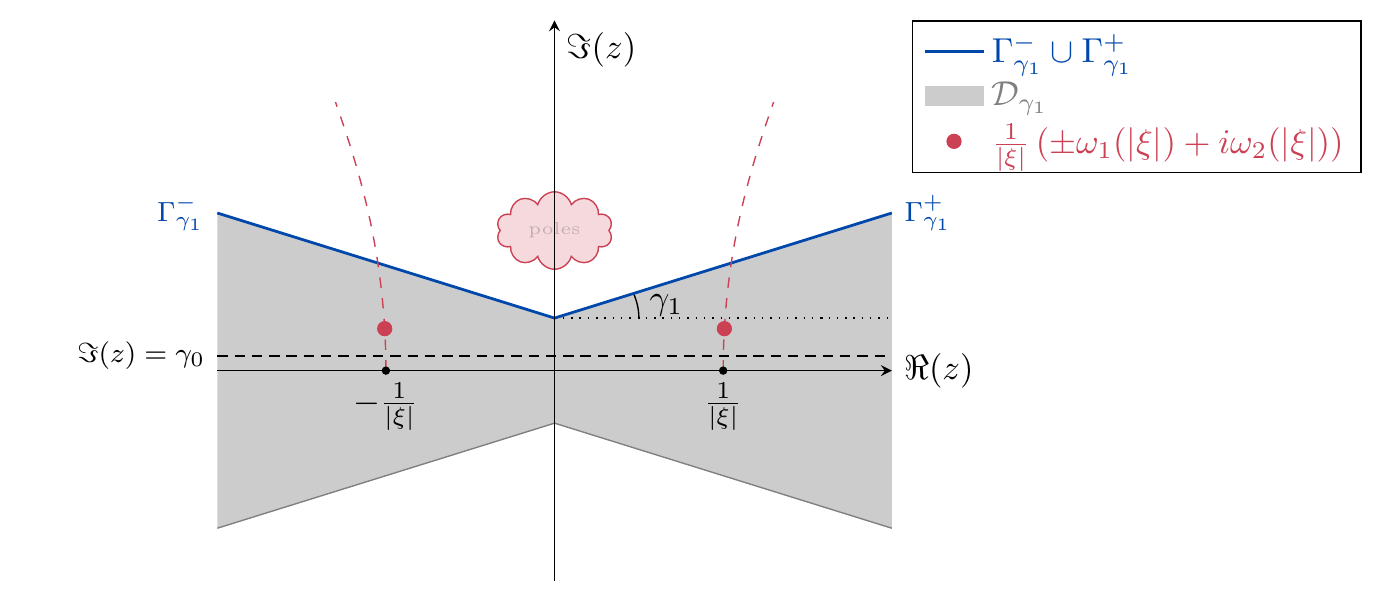}
 \caption{An overview of the integration contours and poles in the proof of Theorem \ref{MainThm}.}
 \label{fig:int_contours}
\end{figure}

In order to prove exponential decay we need to extend the function $\mathbf{k}$ to parts of the upper half plane. For this we recall that $m_0$ extends to an analytic function $m_0:\mathcal{D}_{\vartheta}\to\C$ with $\mathcal{D}_{\vartheta}$ as in \eqref{more10}. Using the Plemelj formula, we see that we can also extend ${\bf k}$ to a holomorphic function on $\mathcal{D}_{\vartheta}$,
\begin{equation}\label{DefkExtended}
\begin{split}
\mathbf{k}(z)&=
\begin{cases}
-\int_{-\infty}^\infty\frac{m_0^\prime(t)}{z-t}dt&\quad\hbox{ if }z\in\mathcal{D}_{\vartheta}\cap\mathbb{H}_-,\\
-\pi\mathcal{H}(m'_0)(x)-i\pi m_0^\prime(x)&\quad\hbox{ if }z=x\in\mathbb{R},\\
-\int_{-\infty}^\infty\frac{m_0^\prime(t)}{z-t}dt-2i\pi m_0^\prime(z)&\quad\hbox{ if }z\in\mathcal{D}_{\vartheta}\cap\mathbb{H}_+.
\end{cases}
\end{split}
\end{equation}

We will need some further bounds.

\begin{lemma}\label{LemG}

The function ${\bf k}$ defined in \eqref{DefkExtended} is analytic in $\mathcal{D}_{\vartheta}$ and satisfies the functional equations
\begin{equation}\label{FunctionalEqG}
\begin{split}
{\bf k}(-\overline{z})=\overline{{\bf k}(z)},\qquad {\bf k}(z)-\overline{{\bf k}(\overline{z})}=-2i\pi m_0^\prime(z),\qquad z\in\mathcal{D}_{\vartheta}.
\end{split}
\end{equation}
Moreover, with ${\bf k}(0)<0$ defined as in \eqref{kof0}, we have the following expansion
\begin{equation}\label{ExpansionG0}
\begin{split}
{\bf k}(z)&={\bf k}(0)-i\pi m_0''(0)z+O(|z|^2),\qquad z\in\mathcal{D}_{\vartheta},\,|z|\leq 1.
\end{split}
\end{equation}

We also have an expansion of $\mathbf{k}$ at $\infty$ in a smaller region: there is $\vartheta'\in(0,\vartheta)$ such that we can decompose
\begin{equation*}
\begin{split}
{\bf k}(z)={\bf k}^{eff}(z)-i\pi m_0^\prime(z),\qquad z\in\mathcal{D}_{\vartheta'}.
\end{split}
\end{equation*}
The function $\mathbf{k}^{eff}:\mathcal{D}_{\vartheta'}\to\C$ is analytic, satisfies the identities
\begin{equation}\label{more2}
{\bf k}^{eff}(z)={\bf k}^{eff}(-z)=\overline{{\bf k}^{eff}(\overline{z})}\qquad\text{ for }z\in\mathcal{D}_{\vartheta'},
\end{equation}
and has the expansion
\begin{equation}\label{ExpansionGInfinityGeneral}
\begin{split}
{\bf{k}}^{eff}(z)-z^{-2}&=O(|z|^{-d-1}+|z|^{-4}\log|z|),\\
|z|^j\big|\partial^j_z\big[{\bf{k}}^{eff}(z)-z^{-2}]\big|&\lesssim_j(|z|^{-d-1}+|z|^{-4}\log|z|),
\end{split}
\end{equation}
if $z\in\mathcal{D}_{\vartheta'}$, $|z|\geq 4$, and $j\geq 1$. In addition, in the thin-tail case we have the more precise expansion
\begin{equation}\label{ExpansionGInfinityTT}
\begin{split}
{\bf{k}}^{eff}(z)-z^{-2}-3a_2z^{-4}&=O_d(|z|^{-d-1}+|z|^{-6}\log|z|),\\
|z|^j\big|\partial^j_z\big[{\bf{k}}^{eff}(z)-z^{-2}-3a_2z^{-4}\big]\big|&\lesssim_{d,j}(|z|^{-d-1}+|z|^{-6}\log|z|),
\end{split}
\end{equation}
where, as before,
\begin{equation*}
a_2=\int_\R t^2m_0(t)\,dt.
\end{equation*}
\end{lemma}

\begin{remark}
The formulas \eqref{ExplicitkGP} show that 
\begin{equation*}
\begin{split}
{\bf k}_1(z)&=\frac{z^2-1}{(1+z^2)^2}+i\frac{2z}{(1+z^2)^2},\qquad{\bf k}_2(z)=\frac{z^4+6z^2-3}{(1+z^2)^3}+i\frac{8z}{(1+z^2)^3},
\end{split}
\end{equation*}
where the first term corresponds to ${\bf k}^{eff}$ and the second term corresponds to $-i\pi m'_0(z)$.
\end{remark}

\begin{proof}[Proof of Lemma \ref{LemG}] The analyticity of the function $\mathbf{k}$ and the identities \eqref{FunctionalEqG} follow directly from the definition \eqref{DefkExtended}. The validity of the expansion \eqref{ExpansionG0} follows as in \eqref{kof0exp}.

To prove the expansion at infinity we define 
\begin{equation}\label{more1}
{\bf k}^{eff}(z):={\bf k}(z)+i\pi m_0^\prime(z).
\end{equation}
The identities \eqref{FunctionalEqG} show that, for any $z\in\mathcal{D}_{\vartheta}$
\begin{equation*}
{\bf k}^{eff}(-\overline{z})-\overline{{\bf k}^{eff}(z)}=0,\qquad {\bf k}^{eff}(z)-\overline{{\bf k}^{eff}(\overline{z})}=0.
\end{equation*}
The identities \eqref{more2} follow. Using the formulas \eqref{DefkExtended}, for $z\in\mathcal{D}_{\vartheta'}\cap\mathbb{H}_-$, $|z|\geq 3$, we calculate
\begin{equation*}
\begin{split}
{\bf k}^{eff}(z)&=-\int_{-\infty}^\infty\frac{m_0^\prime(t)}{z-t}dt+i\pi m_0^\prime(z)=-\int_{0}^\infty\frac{2tm_0^\prime(t)}{z^2-t^2}dt+i\pi m_0^\prime(z).
\end{split}
\end{equation*}
Assume that $z=x+iy$, $x\geq 2$, $-|x|/8\leq y<0$, and decompose
\begin{equation}\label{more4}
\begin{split}
&{\bf k}^{eff}(z)=I_1(z)+I_2(z)+I_3(z),\\
&I_1(z):=\int_{0}^{x/2}\frac{2tm_0^\prime(t)}{t^2-z^2}dt,\qquad I_2(z):=\int_{2x}^\infty\frac{2tm_0^\prime(t)}{t^2-z^2}\,dt,\\
&I_3(z):=\int_{x/2}^{2x}\frac{2tm_0^\prime(t)}{t^2-z^2}dt+i\pi m_0^\prime(z).
\end{split}
\end{equation}
Since $|m'_0(t)|\lesssim (1+t)^{-d-1}$, $\int_0^\infty -2tm'_0(t)\,dt=\int_\R m_0(t)\,dt=1$, and $x\approx |z|$ we can expand
\begin{equation*}
\begin{split}
I_1(z)&=\int_{0}^{x/2}\frac{2tm_0^\prime(t)}{-z^2}dt+O\Big(|z|^{-4}\int_{0}^{x/2}(1+t)^3|m_0^\prime(t)|\,dt\Big)\\
&=\frac{-1}{z^2}\int_{0}^{\infty}2tm_0^\prime(t)\,dt+O\big(|z|^{-4}[\log |z|+|z|^{-d+3}]\big)\\
&=\frac{1}{z^2}+O\big(|z|^{-4}[\log |z|+|z|^{-d+3}]\big),
\end{split}
\end{equation*}
\begin{equation*}
|I_2(z)|\lesssim\int_{2x}^{\infty}\frac{t|m_0^\prime(t)|}{t^2}\,dt\lesssim x^{-d-1}\lesssim |z|^{-d-1},
\end{equation*}
\begin{equation*}
\begin{split}
|I_3(z)|&\lesssim\Big|\int_{x/2}^{2x}\frac{tm_0^\prime(t)}{t+x+iy}\,\,\frac{t-x}{(t-x)^2+y^2}\,dt\Big|+\int_{x/2}^{2x}\frac{t|m_0^\prime(t)|}{x}\,\,\frac{|y|}{(t-x)^2+y^2}\,dt+|m_0^\prime(z)|\\
&\lesssim |z|^{-d-1}.
\end{split}
\end{equation*}
The desired expansion \eqref{ExpansionGInfinityGeneral} follows uniformly when $z\in\mathcal{D}_{\vartheta}\cap\mathbb{H}_-$. We then use the identities \eqref{more2} to prove the validity of the expansion for all $z\in\mathcal{D}_{\vartheta}$. Finally, we use the Cauchy integral formula to prove the bounds in \eqref{ExpansionGInfinityGeneral} for the derivatives $\partial_z^j({\mathbf{k}}^{eff}(z)-z^{-2})$. 

The more precise expansion \eqref{ExpansionGInfinityTT} follows in the same way, by extracting one more term from the integral $I_1(z)$.
\end{proof}

\subsection{Proof of Theorem \ref{MainThm} (i)} The formulas \eqref{Kernels2} and \eqref{DefK} show that
\begin{equation}\label{more5}
\begin{split}
\widehat{G}(\xi,\tau)-\delta_0(\tau)&=-\mathfrak{1}_{[0,\infty)}(\tau)\frac{1}{2\pi}\int_{\R}\frac{K(\xi,\theta)}{1+K(\xi,\theta)}e^{i\theta\tau}\,d\theta\\
&=\mathfrak{1}_{[0,\infty)}(\tau)\frac{|\xi|}{2\pi}\int_{\R}\frac{\mathbf{k}(\theta)}{|\xi|^2-\mathbf{k}(\theta)}e^{i\theta(\tau|\xi|)}\,d\theta.
\end{split}
\end{equation}
For high frequencies $\vert\xi\vert\geq r_0/2$ it follows from Proposition \ref{PropG1} that $\big||\xi|^2-\mathbf{k}(\theta)\big|\gtrsim 1$ for any $\theta\in\R$ and it follows from Lemma \ref{LemG} that $\vert{\bf k}\vert\to0$ as $\vert z\vert\to\infty$ in $\mathcal{D}_{\vartheta^\prime}$. Therefore we can move the contour of integration to the line $\theta\in i\gamma_0+\R$, for $\gamma_0>0$ sufficiently small (to get the exponential decay), and the bounds \eqref{GFHF2} follow easily.

\subsection{Proof of Theorem \ref{MainThm} (ii)}

We consider now the harder case of low frequencies $\vert\xi\vert\ll 1$. We would still like to use the formula \eqref{more5}. However, the function $z\to |\xi|^2-\mathbf{k}(z)$ vanishes at two points $\zeta$ and $-\overline{\zeta}$ close to the real axis (see Figure \ref{fig:scatter} below). Our goal is to move the contour of integration to the boundary of an angular sector in $\mathbb{H}_+$ (in order to gain the exponential factor in front of the error term $\mathcal{E}_l$) and quantify the contributions of the residues.

\begin{figure}[h]
 \centering
 \includegraphics[width=0.5\textwidth]{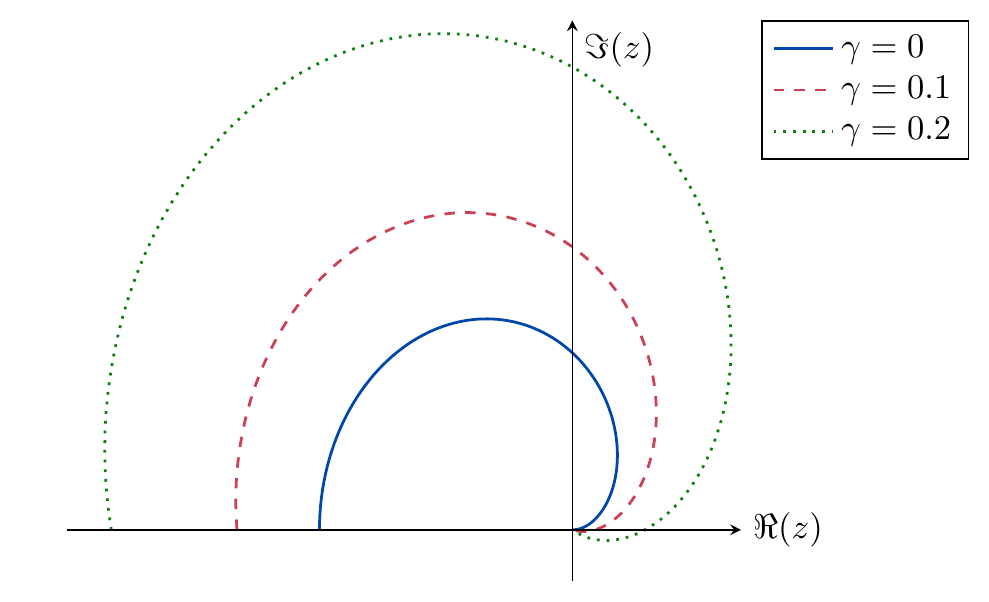}
 \caption{The curve ${\bf k}(\Gamma^+_\gamma) $ for various values of $\gamma$ for the first generalized Poisson equilibrium ${\bf k}_{2}$}
 \label{fig:scatter}
\end{figure}

To make the bounds precise, we use the expansions in Lemma \ref{LemG} and fix a constant $A\geq 1$ such that for any $z\in\mathcal{D}_{\vartheta}$ we have
\begin{equation}\label{more12}
\begin{split}
&|\mathbf{k}(z)|\leq A(1+|z|)^{-2},\\
&|\mathbf{k}(z)-1/z^2|+|m'_0(z)|\leq A[|z|^{-d-1}+|z|^{-4}\log(|z|)],\qquad\text{ if }|\Re z|\geq 4.
\end{split}
\end{equation}
Since $\mathbf{k}(\R)\subseteq\C\setminus [0,\infty)$ (due to Proposition \ref{PropG1}) there is a small constant $\gamma^\ast\in(0,\vartheta/10)$ such that
\begin{equation}\label{more12.5}
|\mathbf{k}(w)-b|\geq\gamma^\ast\,\,\text{ for any }\,\,b\in[0,\infty)\,\,\text{ and }\,\,w\in\{z\in\C:\,|\Re z|\leq 10A,\,|\Im z|\leq\gamma^\ast\}.
\end{equation}

We will prove the following lemma:

\begin{lemma}\label{LemContourSmallXi}
(i) There are small numbers $\gamma_1,r_1>0$ (depending only on the constants $A$ and $\gamma^\ast$ fixed earlier) such that if $r\in(0,r_1]$ then the function 
\begin{equation}\label{more11}
z\to r^2-\mathbf{k}(z)
\end{equation}
vanishes (of order $1$) at exactly two points $\zeta(r)$ and $-\overline{\zeta(r)}$ in the region $\mathcal{D}_{\gamma_1}$. Moreover
\begin{equation}\label{more11.5}
|r^2-\mathbf{k}(z)|\gtrsim r^2+|z|^{-2}\qquad\text { if }z\in\mathcal{D}_{2\gamma_1}\setminus\mathcal{D}_{\gamma_1/2}.
\end{equation}

(ii) For any $r\in(0,r_1]$ and $a\geq 1$ we have
\begin{equation}\label{more11.1}
|\zeta(r)-1/r|\lesssim r^{d-2}+r\log(1/r)\qquad\text {and }\qquad r^a\big|\partial_r^a[\zeta(r)-1/r]\big|\lesssim_a r^{d-2}+r\log(1/r).
\end{equation}
Moreover, in the thin-tail case we have the more precise expansion
\begin{equation}\label{more11.3}
\begin{split}
&|\zeta(r)-1/r-3a_2r/2|\lesssim_d r^{d-2}+r^3\log(1/r),\\
&r^a\big|\partial_r^a[\zeta(r)-1/r-3a_2r/2]\big|\lesssim_{d,a} r^{d-2}+r^3\log(1/r),\qquad a\geq 1,
\end{split}
\end{equation}
where $a_2$ is defined as in \eqref{a2def}.

(iii) In addition, letting $\zeta=\zeta_1+i\zeta_2$, we have
\begin{equation}\label{more11.2}
|\zeta_1(r)-1/r|\lesssim r^{d-2}+r\log(1/r),\qquad \Big|\zeta_2(r)+\frac{\pi m'_0(\zeta_1(r))}{2r^3}\Big|\lesssim (r+r^{d-1})|\zeta_2(r)|.
\end{equation}
In particular, $\zeta_2(r)\approx -m'_0(\zeta_1(r))r^{-3}>0$. 
\end{lemma}

\begin{figure}[h]
 \centering
 \includegraphics[width=0.6\textwidth]{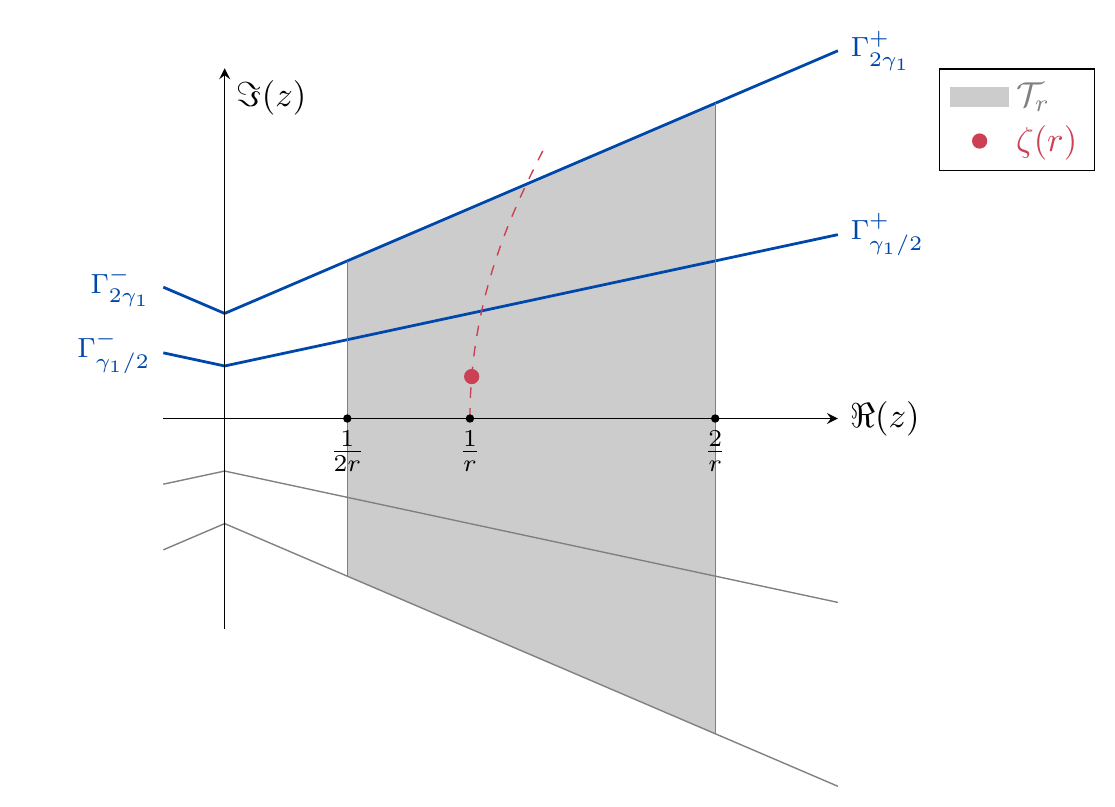}
 \caption{The trapezoid $\mathcal{T}_r$ used via Rouch\'e's Theorem in the proof of Lemma \ref{LemContourSmallXi}.}
 \label{fig:Rouche_contours}
\end{figure}

\begin{proof}  (i) The implied constants in this lemma may depend only on the constants $d-1>0$, $A$, and $\gamma^\ast$. We set $\gamma_1:=\gamma^\ast(1+10A)^{-1}$ and $r_1=r_1(A,\gamma^\ast)$ sufficiently small. 

Using \eqref{more12} we have
\begin{equation}\label{more13}
|r^2-\mathbf{k}(z)|\gtrsim r^2+|z|^{-2}\qquad\text{ if }r\in (0,r_1],\,z\in\mathcal{D}_{2\gamma_1},\text{ and }|\Re z|\notin[1/(2r),2/r].
\end{equation}

To analyze the remaining region we would like to use Rouch\'{e}'s theorem in the trapezoid $\mathcal{T}_r:=\{z\in\mathcal{D}_{2\gamma_1}:\,\Re z\in[1/(2r),2/r]\}$ -- see Figure \ref{fig:Rouche_contours}. Let $g(z):=r^2-z^{-2}$. Using \eqref{more12} we have
\begin{equation*}
\big|g(z)-(r^2-\mathbf{k}(z))\big|\leq A\big(|z|^{-3}+|z|^{-d-1}\big)\leq 20A(r^3+r^{d+1}),\qquad \text{ if }z\in\mathcal{T}_r.
\end{equation*}
On the other hand, if $z\in\mathcal{D}_{2\gamma_1}\setminus\mathcal{D}_{\gamma_1/2}$ then
\begin{equation*}
|g(z)|\geq r|r-z^{-1}|\geq (r^2/4)|rz-1|\geq \gamma_1r^2/20.
\end{equation*}
The function $g$ vanishes (of order $1$) at exactly one point $z=1/r\in\mathcal{T}_r$. Therefore, the conclusions in part (i) follow using  Rouch\'{e}'s theorem in the trapezoid $\mathcal{T}_r$, the bounds \eqref{more13}, and the identity $\mathbf{k}(-\overline{z})=\overline{\mathbf{k}(z)}$ in \eqref{FunctionalEqG}.

(ii) We analyze now the dispersion curve $\zeta(r)\in\mathcal{T}_r$ defined by the identity 
\begin{equation}\label{more14}
r^2=\mathbf{k}(\zeta(r)).
\end{equation}
To prove the estimates \eqref{more11.1} we decompose 
\begin{equation*}
\zeta(r)=(1+\delta(r))/r,\qquad {\bf k}(z)=z^{-2}(1+L(z)). 
\end{equation*}
The equation ${\bf k}(\zeta(r))=r^2$ becomes
\begin{equation}\label{more19.2}
2\delta(r)+\delta(r)^2=L((1+\delta(r))/r).
\end{equation}
The function $L$ is analytic in the region $\{z\in\mathcal{D}_{\vartheta'}:\,|\Re z|\geq 10A\}$ and satisfies
\begin{equation}\label{more19.3}
|z|^j\big|\partial_z^jL(z)\big|\lesssim_j |z|^{-d+1}+|z|^{-2}\log|z|,\qquad j\geq 0,
\end{equation}
in this region, using Lemma \ref{LemG} and the Cauchy integral formula. A simple fixed point argument then shows that the equation \eqref{more19.2} admits a unique solution $\delta(r)$ with $|\delta(r)|\lesssim r^{d-1}+r^2\log (1/r)$. The full estimates \eqref{more11.1} then follow using again \eqref{more19.2}--\eqref{more19.3} and taking $r\partial_r$ derivatives.

The bounds \eqref{more11.3} are similar. Indeed, in the thin tail case, we decompose 
\begin{equation*}
\zeta(r)=\frac{1+3a_2r^2/2+\delta_2(r)}{r}\qquad {\bf k}(z)=z^{-2}(1+3a_2z^{-2}+L_2(z)),
\end{equation*}
and the equation ${\bf k}(\zeta(r))=r^2$ becomes
\begin{equation}\label{more19.5}
2\delta_2(r)+3a_2r^2\delta_2(r)+\delta_2(r)^2+9a_2^2r^4/4=L_2(\zeta(r))+3a_2[\zeta(r)^{-2}-r^2].
\end{equation}
The function $L_2$ is analytic in the region $\{z\in\mathcal{D}_{\vartheta}:\,|\Re z|\geq 10A\}$ and satisfies
\begin{equation}\label{more19.6}
|z|^j\big|\partial_z^jL_2(z)\big|\lesssim_j |z|^{-d+1}+|z|^{-4}\log|z|,\qquad j\geq 0,
\end{equation}
in this region, due to \eqref{ExpansionGInfinityTT}. Since we already know that $|\delta_2(r)|\lesssim r$ it follows from \eqref{more19.5} that $|\delta_2(r)|\lesssim_d r^{d-1}+r^4\log(1/r)$. The bounds \eqref{more11.3} follow from \eqref{more19.5}--\eqref{more19.6}, by taking $r\partial_r$ derivatives.

(iii) To prove \eqref{more11.2} we start from the defining identity \eqref{more14} and rewrite it in the form
\begin{equation*}
\begin{split}
r^2&={\bf k}^{eff}(\zeta_1(r)+i\zeta_2(r))-i\pi m'_0(\zeta_1(r)+i\zeta_2(r))\\
&={\bf k}^{eff}(\zeta_1(r))+i\zeta_2(r)\partial_z{\bf k}^{eff}(\zeta_1(r))-i\pi m'_0(\zeta_1(r))+O(|\zeta_2(r)|(r^4+r^{d+2}),
\end{split}
\end{equation*}
using the decomposition ${\bf k}={\bf k}^{eff}-i\pi m'_0$ in Lemma \ref{LemG} and the bounds $|\zeta_2(r)|\lesssim 1+r^{d-2}$ proved in part (ii). Taking real and imaginary parts, it follows that
\begin{equation}\label{more22}
\begin{split}
&\big|r^2-{\bf k}^{eff}(\zeta_1(r))\big|\lesssim |\zeta_2(r)|(r^4+r^{d+2}),\\
&\big|\zeta_2(r)\partial_z{\bf k}^{eff}(\zeta_1(r))-\pi m'_0(\zeta_1(r))\big|\lesssim |\zeta_2(r)|(r^4+r^{d+2}).
\end{split}
\end{equation}
We can now use the expansion ${\bf k}^{eff}(\zeta_1(r))=1/(\zeta_1(r))^{-2}+O(r^{d+1}+r^4\log(1/r))$ in \eqref{ExpansionGInfinityGeneral} to prove the bounds on $|\zeta_1(r)-1/r|$ in \eqref{more11.2}. Then we expand $\partial_z{\bf k}^{eff}(\zeta_1(r))=-2r^3+O(r^4+r^{d+2})$ to derive the  bounds on $\zeta_2(r)$ in \eqref{more11.2}.
\end{proof}

\subsubsection{Conclusion of the proof}
We are now ready to complete the proof of the first main theorem. We start from the formula \eqref{more5}, transfer the contour of integration to the lines $\Gamma^{-}_{\gamma_1}\cup\Gamma^{+}_{\gamma_1}$ and use the residue formula (see Figure \ref{fig:int_contours}). The result is
\begin{equation}\label{more30}
\begin{split}
&\widehat{G}(\xi,\tau)-\delta_0(\tau)=\mathfrak{1}_{[0,\infty)}(\tau)\frac{|\xi|}{2\pi}\int_{\R}\frac{\mathbf{k}(\theta)}{|\xi|^2-\mathbf{k}(\theta)}e^{i\theta(\tau|\xi|)}\,d\theta\\
&=\mathfrak{1}_{[0,\infty)}(\tau)|\xi|\Big\{\frac{1}{2\pi}\int_{\Gamma^{-}_{\gamma_1}\cup\Gamma^{+}_{\gamma_1}}\frac{\mathbf{k}(\theta)}{|\xi|^2-\mathbf{k}(\theta)}e^{i\theta(\tau|\xi|)}\,d\theta+i\,\mathrm{res}_{\zeta(|\xi|)}H_{|\xi|,\tau}+i\,\mathrm{res}_{-\overline{\zeta(|\xi|)}}H_{|\xi|,\tau}\Big\}\\
&=I_1(|\xi|,\tau)+I_2(|\xi|,\tau),
\end{split}
\end{equation}
where
\begin{equation}\label{more31}
\begin{split}
&H_{r,\tau}(z):=\frac{\mathbf{k}(z)}{r^2-\mathbf{k}(z)}e^{iz(\tau r)},\\
&I_1(r,\tau)=\mathfrak{1}_{[0,\infty)}(\tau)\frac{r}{2\pi}\int_{\Gamma^{-}_{\gamma_1}\cup\Gamma^{+}_{\gamma_1}}\frac{\mathbf{k}(\theta)}{r^2-\mathbf{k}(\theta)}e^{i\theta(\tau r)}\,d\theta,\\
&I_2(r,\tau):= \mathfrak{1}_{[0,\infty)}(\tau)r\big[i\mathrm{res}_{\zeta(r)}H_{r,\tau}+i\mathrm{res}_{-\overline{\zeta(r)}}H_{r,\tau}\big].
\end{split}
\end{equation}

{\bf{Step 1.}} The function $I_1$ generates the error term $\mathcal{E}_l$ in \eqref{GFLF}. To obtain additional smallness as $r\to 0$ we notice that
\begin{equation*}
 \int_{\Gamma^{-}_{\gamma_1}\cup\Gamma^{+}_{\gamma_1}}\frac{\theta^{-2}}{r^2-\theta^{-2}}e^{i\theta(\tau r)}\,d\theta\equiv 0
\end{equation*}
if $\tau\geq 0$ and $r>0$. This is because the integrand is analytic in $\mathbb{H}_+$, so one can freely move the contour of integration to infinity. Therefore, as in \cite{BeMaMoLin},
\begin{equation}\label{more31.5}
\begin{split}
I_1(r,\tau)&=\mathfrak{1}_{[0,\infty)}(\tau)\frac{r}{2\pi}\int_{\Gamma^{-}_{\gamma_1}\cup\Gamma^{+}_{\gamma_1}}\Big[\frac{\mathbf{k}(\theta)}{r^2-\mathbf{k}(\theta)}-\frac{\theta^{-2}}{r^2-\theta^{-2}}\Big]e^{i\theta(\tau r)}\,d\theta\\
&=\mathfrak{1}_{[0,\infty)}(\tau)\frac{r}{2\pi}\int_{\Gamma^{-}_{\gamma_1}\cup\Gamma^{+}_{\gamma_1}}\frac{r^2[\mathbf{k}(\theta)-\theta^{-2}]}{(r^2-\mathbf{k}(\theta))(r^2-\theta^{-2})}e^{i\theta(\tau r)}\,d\theta.
\end{split}
\end{equation}

We define the function 
\begin{equation}\label{more40}
Q(r,z):=\frac{r^2[\mathbf{k}(z)-z^{-2}]}{(r^2-\mathbf{k}(z))(r^2-z^{-2})}.
\end{equation}
The function is analytic in $z$ in the region $\mathcal{D}_{2\gamma_1}\setminus\mathcal{D}_{\gamma_1/2}$ and satisfies the bounds
\begin{equation}\label{more41}
\begin{split}
|Q(r,z)|&\lesssim \frac{r^2[(1+|z|)^{-d-1}+(1+|z|)^{-4}\log (2+|z|)]}{(r^2+|z|^{-2})(r^2+|z|^{-2})}\\
&\lesssim \frac{r^2[(1+|z|)^{-d+3}+\log (2+|z|)]}{(1+r^2|z|^{2})^2}.
\end{split}
\end{equation}
Taking $r\partial_r$ or $z\partial_z$ derivatives preserves these bounds in a slightly smaller region (due to the Cauchy integral formula),
\begin{equation}\label{more42}
|(r\partial_r)^a(z\partial_z)^bQ(r,z)|\lesssim_{a,b}\frac{r^2[(1+|z|)^{-d+3}+\log (2+|z|)]}{(1+r^2|z|^{2})^2},
\end{equation}
for $r\in(0,r_1]$ and $z\in\mathcal{D}_{3\gamma_1/2}\setminus\mathcal{D}_{2\gamma_1/3}$.

We can now estimate the integrals $I_1$. We use \eqref{more31.5} and \eqref{more41} and notice that $|e^{i\theta(\tau r)}|\lesssim e^{-\gamma_1\tau r}$ for $\theta\in\Gamma^{-}_{\gamma_1}\cup\Gamma^{+}_{\gamma_1}$. Thus
\begin{equation*}
\begin{split}
|I_1(r,\tau)|&\lesssim \mathfrak{1}_{[0,\infty)}(\tau)e^{-\gamma_1\tau r}r\int_{\R}\frac{r^2[(1+|x|)^{-d+3}+\log (2+|x|)]}{(1+r^2|x|^{2})^2}\,dx\\
&\lesssim \mathfrak{1}_{[0,\infty)}(\tau)e^{-\gamma_1\tau r}\cdot [r^2\log (1/r)+r^{d-1}].
\end{split}
\end{equation*}
Taking $r\partial_r$ and $\tau\partial_\tau$ derivatives does not change the bounds, since one can use \eqref{more42} instead of \eqref{more41} and integrate by parts in $\theta$ when the derivative hits the exponential $e^{i\theta(\tau r)}$. This completes the proof for the bounds on the $\mathcal{E}_l$ term in \eqref{GFLF2}.

The analysis in the thin-tail case is similar, except that we correct with the function $\theta^{-2}+3a_2\theta^{-4}$ instead of the  function $\theta^{-2}$, leading to the stronger bounds in \eqref{GFLF4}.

{\bf{Step 2.}} We examine now the function $I_2$. With $\zeta=\zeta_1+i\zeta_2$ we calculate
\begin{equation*}
\mathrm{res}_{\zeta(r)}H_{r,\tau}=\frac{\mathbf{k}(\zeta(r))}{-\mathbf{k}'(\zeta(r))}e^{i\tau r\zeta(r)}=\frac{-\mathbf{k}(\zeta(r))}{\mathbf{k}'(\zeta(r))}e^{i\tau r\zeta_1(r)}e^{-\tau r\zeta_2(r)}.
\end{equation*}
In view of \eqref{FunctionalEqG}, we have $\mathbf{k}(-\overline{z})=\overline{\mathbf{k}(z)}$ and $\mathbf{k}'(-\overline{z})=-\overline{\mathbf{k}'(z)}$. Therefore
\begin{equation*}
\mathrm{res}_{-\overline{\zeta(r)}}H_{r,\tau}=\frac{\mathbf{k}(-\overline{\zeta(r)})}{-\mathbf{k}'(-\overline{\zeta(r)})}e^{-i\tau r\overline{\zeta(r)}}=\frac{\overline{\mathbf{k}(\zeta(r))}}{\overline{\mathbf{k}'(\zeta(r))}}e^{-i\tau r\zeta_1(r)}e^{-\tau r\zeta_2(r)}.
\end{equation*}
Using the definition \eqref{more31} we have
\begin{equation*}
I_2(r,\tau)= \mathfrak{1}_{[0,\infty)}(\tau)\Re\Big\{i\frac{-2r\mathbf{k}(\zeta(r))}{\mathbf{k}'(\zeta(r))}e^{i\tau r\zeta_1(r)}e^{-\tau r\zeta_2(r)}\Big\}=\mathfrak{1}_{[0,\infty)}(\tau)\Re\big\{i(1+\mathfrak{m}_l(r))e^{i\tau \omega(r)}\big\},
\end{equation*}
where
\begin{equation}\label{more32}
\omega(r):=r\zeta(r)=r\zeta_1(r)+ir\zeta_2(r),\qquad \mathfrak{m}_l(r):=\frac{-2r\mathbf{k}(\zeta(r))}{\mathbf{k}'(\zeta(r))}-1.
\end{equation}
This agrees with the second term in the main identity \eqref{GFLF}. The bounds \eqref{GFLF3} and \eqref{GFLF5} follow from \eqref{more11.1}--\eqref{more11.2}. The bounds on $\mathfrak{m}_l$ in \eqref{GFLF2} and \eqref{GFLF4} follow using also the expansions in Lemma \ref{LemG}. This completes the proof of Theorem \ref{MainThm}.

\section{Nonlinear asymptotic stability: proof of Theorem \ref{thm:main_simple}}\label{outlineN}

In this section we discuss some of the main ideas in the proof of Theorem \ref{thm:main_simple} regarding stability of the Poisson equilibrium \eqref{NVP.2}, following our paper \cite{IoPaWaWiPoisson}. 

\subsection{The main decomposition and the bootstrap argument} We recall first our main equations. The perturbation $f$ satisfies the Vlasov-Poisson system
\begin{equation}\label{cunu1}
\begin{split}
&\left(\partial_t+v\cdot\nabla_x\right)f+E\cdot\nabla_vM_1+E\cdot\nabla_vf=0,\\
&E:=\nabla_x\Delta_x^{-1}\rho,\qquad \rho(x,t):=\int_{\mathbb{R}^3}f(x,v,t)dv.
\end{split}
\end{equation}
We introduce the ``backwards characteristics'' of this system, which are the functions $X,V:\R^3\times\R^3\times\mathcal{I}^2_T\to\R^3$ obtained by solving the ODE system
\begin{equation}\label{cunu2}
\begin{alignedat}{2}
&\partial_sX(x,v,s,t)=V(x,v,s,t),\qquad &X(x,v,t,t)&=x,\\
&\partial_sV(x,v,s,t)=E(X(x,v,s,t),s),\qquad &V(x,v,t,t)&=v,
\end{alignedat}
\end{equation}
where $\mathcal{I}^2_T=\{(s,t)\in[0,T]^2:\,s\leq t\}$. These equations \eqref{Lan1} yield the reproducing formulas
\begin{equation}\label{cunu3}
\begin{split}
X(x,v,s,t)&=x-(t-s)v+\int_s^t(\tau-s)E(X(x,v,\tau,t),\tau)\,d\tau,\\
V(x,v,s,t)&=v-\int_{s}^t E(X(x,v,\tau,t),\tau)\,d\tau.
\end{split}
\end{equation}
The main equation \eqref{cunu1} gives
\begin{equation}\label{cunu4}
\rho(x,t)+\int_0^t\int_{\R^3}(t-s)\rho(x-(t-s)v,s)M_1(v)\,dvds=\mathcal{N}(x,t),
\end{equation}
for any $(x,t)\in\R^3\times[0,T]$, where
\begin{equation}\label{cunu5}
\begin{split}
\mathcal{N}(x,t)&:=\mathcal{N}_1(x,t)+\mathcal{N}_2(x,t),\\
\mathcal{N}_1(x,t)&:=\int_{\R^3}f_0(X(x,v,0,t),V(x,v,0,t))\,dv,\\
\mathcal{N}_2(x,t)&:=\int_0^t\int_{\R^3}\big\{E(x-(t-s)v,s)\cdot \nabla_vM_1(v)\\
&\qquad\qquad-E(X(x,v,s,t),s)\cdot \nabla_vM_1(V(x,v,s,t))\big\}\,dvds.
\end{split}
\end{equation}

In our case, when $M_1$ is the Poisson equilibrium, the Volterra equation \eqref{cunu5} can be solved explicitly, using the Green's function $G=G_1$ calculated in \eqref{GP}. We have
\begin{equation}\label{cunu6}
 \widehat{\rho}(\xi,t)=\widehat{\mathcal{N}}(\xi,t)-\int_0^t\widehat{\mathcal{N}}(\xi,\tau)e^{-(t-\tau)|\xi|}\sin(t-\tau)\,d\tau.
\end{equation}
We examine the definitions \eqref{cunu5} and write $\mathcal{N}$ in the form
\begin{equation}\label{cunu7}
\mathcal{N}(x,t)=\int_{\R^3}h(x-tv,v,t)\,dv,
\end{equation}
for a suitable real-valued function $h$. For nonlinear analysis we need two alternative formulas:

\begin{lemma}[\protect{\cite[Lemma 2.1]{IoPaWaWiPoisson}}]\label{SolvingVolterraP}

With the notation above, if $f$ is a regular solution of the Vlasov-Poisson system \eqref{cunu1} on the time interval $[0,T]$ then we have two alternative expressions:
\begin{equation*}
\begin{split}
\rho(x,t)&=R^\ast[h](x,t)+\Re\left\{e^{-it}T^\ast[h](x,t)\right\},\qquad \ast\in\{I,II\},
\end{split}
\end{equation*}
where the linear operators $R^I,T^I,R^{II},T^{II}$ are defined by
\begin{equation}\label{FormulaRIP}
\begin{aligned}
 \widehat{R^{I}[h]}(\xi,t)&:=\int_{\mathbb{R}^3} e^{-it\langle v,\xi\rangle}\widehat{h}(\xi,v,t)dv,\\
 \widehat{T^{I}[h]}(\xi,t)&:=-i\int_{0}^t\int_{\mathbb{R}^3}e^{is}e^{-(t-s)\vert\xi\vert}e^{-is\langle v,\xi\rangle}\widehat{h}(\xi,v,s)dvds,\\
\end{aligned}
\end{equation}
and
\begin{equation}\label{FormulaRIIP}
\begin{aligned}
 \widehat{R^{II}[h]}(\xi,t)&:=\int_{\mathbb{R}^3} \frac{(\vert\xi\vert-i\langle v,\xi\rangle)^2}{1+(\vert\xi\vert-i\langle v,\xi\rangle)^2}e^{-it\langle \xi,v\rangle}\widehat{h}(\xi,v,t)dv,\\
 \widehat{T^{II}[h]}(\xi,t)&:=e^{-t\vert\xi\vert}\int_{\mathbb{R}^3}\frac{1}{1-\langle v,\xi\rangle-i\vert\xi\vert}\widehat{h}(\xi,v,0)dv\\
 &\quad+\int_{0}^t\int_{\mathbb{R}^3}\frac{e^{is}}{1-\langle v,\xi\rangle-i\vert\xi\vert}e^{-(t-s)\vert\xi\vert}e^{-is\langle v,\xi\rangle}\widehat{\partial_s h}(\xi,v,s)dvds.
\end{aligned} 
\end{equation}
\end{lemma}

The two representations follow easily from \eqref{cunu6}--\eqref{cunu7}, using the identity \eqref{NFG} and integration by parts in $s$ for the second representation.

Lemma \ref{SolvingVolterraP} leads to a natural decomposition of the density $\rho$ into a ``static'' component $\rho^{stat}$ and an ``oscillatory'' component $\Re(e^{-it}\rho^{osc})$. The corresponding decomposition of the electric field claimed in \eqref{eq:E-decomp} follows by setting $E^\ast:=\nabla\Delta^{-1}\rho^\ast$, $\ast\in\{osc,stat\}$. We will control both of these components using a bootstrap argument and different norms. 

The more subtle issue is whether to use the first representation \eqref{FormulaRIP} or the second representation \eqref{FormulaRIIP} to recover the density $\rho$. Ideally, we would like to use the second representation, for two reasons: (1) the static term $R^{II}$ contains an additional favorable factor at low frequencies, compared to the term $R^I$, and (2) the derivative $\partial_sh$ of the ``profile" $h$ is expected to be smaller than the profile itself. 

However, the second representation contains the potentially small denominator $1-\langle v,\xi\rangle-i\vert\xi\vert$, coming from the normal form. To avoid the associated singularity, our basic idea is to use the first representation when this denominator is small (essentially $\big|1-\langle v,\xi\rangle-i\vert\xi\vert\big|\lesssim 1$) and then use the second representation when the denominator is large.

\subsubsection{Dyadic decomposition of the nonlinearity} To implement this idea, we need to decompose dyadically the nonlinear terms in \eqref{cunu5} in frequency, velocity space, and time. For this we fix an even smooth function $\varphi: \R\to[0,1]$ supported in $[-8/5,8/5]$ and equal to $1$ in $[-5/4,5/4]$,
and define
\begin{equation}\label{phik*}
\varphi_k(x) := \varphi(|x|/2^k) - \varphi(|x|/2^{k-1}) , \quad \varphi_{\leq k}(x):=\varphi(|x|/2^k),
  \quad\varphi_{\geq k}(x) := 1-\varphi(x/2^{k-1}),
\end{equation}
for any $k\in\mathbb{Z}$ and $x\in\R^d$, $d\geq 1$. Let $P_k$, $P_{\leq k}$, and $P_{\geq k}$ denote the operators on $\R^3$ defined by the Fourier multipliers $\varphi_k$, $\varphi_{\leq k}$, and $\varphi_{\geq k}$ respectively, and let $\widetilde{\varphi}_0=\varphi_{\leq 0}$ and $\widetilde{\varphi}_j=\varphi_{j}$ if $j\geq 1$.

We examine the nonlinearities $\mathcal{N}_1$ and $\mathcal{N}_2$ and define the functions $L_{1,j}:\R^3\times\R^3\times[0,T]\to\R$ and $L_{2,j}:\R^3\times\R^3\times\mathcal{I}_T^2\to\R$  by
\begin{equation}\label{qwp1}
\begin{split}
L_{1,j}(x,v,t)&:=\widetilde{\varphi}_j(v)\cdot f_0(X(x+tv,v,0,t),V(x+tv,v,0,t)),\\
L_{2,j}(x,v,s,t)&:=E(x,s)\cdot M'_j(v)-E(X(x+(t-s)v,v,s,t),s)\cdot M'_j(V(x+(t-s)v,v,s,t)),
\end{split}
\end{equation} 
where $j\in\Z_+$ and $M'_j(v):=\widetilde{\varphi}_j(v)\nabla_vM_1(v)$. Then for any $j\in\Z_+$ and $k\in\Z$ we define the functions $L_{1,j,k}:\R^3\times\R^3\times[0,T]\to\R$ and $L_{2,j,k}:\R^3\times\R^3\times\mathcal{I}_T^2\to\R$ by
\begin{equation}\label{qwp2}
L_{1,j,k}(x,v,t):=P_kL_{1,j}(x,v,t),\qquad L_{2,j,k}(x,v,s,t):=P_kL_{2,j}(x,v,s,t),
\end{equation} 
where the projections $P_k$ apply in the $x$ variable. It follows from \eqref{cunu5} that
\begin{equation}\label{Lan10}
\begin{split}
&\mathcal{N}_1=\sum_{j\in\Z_+,\,k\in\Z}\mathcal{N}_{1,j,k},\qquad\mathcal{N}_{1,j,k}(x,t):=\int_{\R^3}L_{1,j,k}(x-tv,v,t)\,dv,\\
&\mathcal{N}_2=\sum_{j\in\Z_+,\,k\in\Z}\mathcal{N}_{2,j,k},\qquad\mathcal{N}_{2,j,k}(x,t):=\int_0^t\int_{\R^3}L_{2,j,k}(x-(t-s)v,v,s,t)\,dvds.
\end{split}
\end{equation}
Finally, we define
\begin{equation}\label{cunu9}
\begin{split}
&h_{1,j,k}(x,v,t):=L_{1,j,k},\qquad h_{2,j,k}(x,v,t):=\int_0^tL_{2,j,k}(x+sv,v,s,t)\,ds,\\
&R^\ast_{a,j,k}:=R^\ast [h_{a,j,k}],\qquad T^\ast_{a,j,k}:=T^\ast [h_{a,j,k}],
\end{split}
\end{equation}
for $a\in\{1,2\}$, $\ast\in\{I,II\}$, $j\in\Z_+$, $k\in\Z$, where the operators $R^{I}, T^{I}, R^{II}, T^{II}$ are defined in \eqref{FormulaRIP}--\eqref{FormulaRIIP}. 

We are now ready to complete our main decomposition. We define the sets
\begin{equation}\label{sug30}
\begin{split}
A^I&:=\big\{(j,k,m)\in\Z_+\times\Z\times\Z_+:\,m<\delta^{-4}\text{ or }j>19m/20\text{ or }k+j+\delta m/3> 0\big\},\\
A^{II}&:=\big\{(j,k,m)\in\Z_+\times\Z\times\Z_+:\,m\geq\delta^{-4}\text{ and }j\leq 19m/20\text{ and }k+j+\delta m/3\leq 0\big\},
\end{split}
\end{equation}
where $\delta\in(0,1/100]$ is a small parameter. We have thus established the following full decomposition of the density $\rho$:

\begin{lemma}[\protect{\cite[Corollary 2.3]{IoPaWaWiPoisson}}]\label{rhodeco}
Assume that $f:\R^3\times\R^3\times[0,T]\to\R$ is a regular solution of the system \eqref{VP} and define the function $\rho$ as before. Then we can decompose
\begin{equation}\label{sug31}
\rho=\rho^{stat}+\Re\big\{e^{-it}\rho^{osc}\big\},
\end{equation}
where, with the definitions above,
\begin{equation}\label{bvn0}
\begin{split}
\rho^{stat}&=\rho^{stat}_{1,I}+\rho^{stat}_{1,II}+\rho^{stat}_{2,I}+\rho^{stat}_{2,II},\\
\rho^{osc}&=\rho^{osc}_{1,I}+\rho^{osc}_{1,II}+\rho^{osc}_{2,I}+\rho^{osc}_{2,II},\\
\end{split}
\end{equation}
where for $a\in\{1,2\}$ and $\ast\in\{I,II\}$ we define 
\begin{equation}\label{bvn1}
\rho^{stat}_{a,\ast}(x,t):=\sum_{(j,k,m)\in A^\ast}\widetilde{\varphi}_m(t)R^{\ast}_{a,j,k}(x,t),
\end{equation}
\begin{equation}\label{bvn2}
\rho^{osc}_{a,\ast}(x,t):=\sum_{(j,k,m)\in A^\ast}\widetilde{\varphi}_m(t)T^{\ast}_{a,j,k}(x,t).
\end{equation}
\end{lemma}

The main point of this decomposition is that we are able to prove stronger control of the low frequencies of the stationary component $\rho^{stat}$ compared to the time-oscillatory component $\rho^{osc}$. See Proposition \ref{MainBootstrapProp} below.

\subsubsection{Norms and the bootstrap proposition}\label{bootstrapass} 

We are now ready to define our main norms and state the main bootstrap proposition. Let $B_T$ denote the space of continuous functions on $\R^3\times [0,T]$ defined by the norm
\begin{equation}\label{normA}
\begin{split}
&\|f\|_{B_T}:=\sup_{t\in[0,T]}\|f(t)\|_{B^0_t},\\
&\|f(t)\|_{B^0_t}:=\sup_{k\in\Z}\big\{\langle t\rangle^3\|P_kf(t)\|_{L^\infty}+\|P_kf(t)\|_{L^1}\big\}.
\end{split}
\end{equation}
Assume that $\delta\in( 0,1/100]$ is a small parameter and define the norms
\begin{equation}\label{may12eqn21}
\begin{split}
\Vert f\Vert_{Stat_\delta}&:=\Vert  \langle t\rangle^{1-2\delta}\langle \nabla_x\rangle      f(t)\Vert_{B_T},\\
\Vert f\Vert_{Osc_\delta}&:=\Vert  \langle t\rangle^{-\delta}f(t)\Vert_{B_T}+\Vert   \langle t\rangle^{1-2\delta}\nabla_{x,t}f(t)\Vert_{B_T}.
\end{split}
\end{equation}

We are now ready to state our main bootstrap proposition:
\begin{proposition}[\protect{\cite[Proposition 2.4]{IoPaWaWiPoisson}}]\label{MainBootstrapProp} 
There exist $0<\overline{\varepsilon}\ll 1$, $\delta\in (0,\frac{1}{100}]$ such that the following is true:

Assume that $f:\R^3\times\R^3\times[0,T]\to\R$ is a regular solution of the system \eqref{VP} for some $T>0$ with initial data $f(0)=f_0:\R^3\times\R^3\to\R$ satisfying the smallness condition
\begin{equation}\label{bootinit}
  \sum_{\abs{\alpha}+\abs{\beta}\leq 1}\norm{\ip{v}^{4.5}\partial_x^\alpha\partial_v^\beta f_0(x,v)}_{L^\infty_x L^\infty_v}+\norm{\ip{v}^{4.5}\partial_x^\alpha\partial_v^\beta f_0(x,v)}_{L^1_x L^\infty_v}\leq\varep_0\leq \bar\varep.
  \end{equation}
Assume the associated density $\rho$ decomposes as $\rho=\rho^{stat}+\Re\{e^{-it}\rho^{osc}\}$ as in Lemma \ref{rhodeco}, such that
\begin{equation}\label{YW12}
\Vert \rho^{stat}\Vert_{Stat_\delta}+\Vert \rho^{osc}\Vert_{Osc_\delta}\leq\varep_1,
\end{equation}
for some $0<\varep_1\leq\varep_0^{3/4}$. Then the functions $\rho^{stat}$ and $\rho^{osc}$ satisfy the improved bounds
\begin{equation}\label{YW12improved}
\Vert \rho^{stat}\Vert_{Stat_\delta}+\Vert \rho^{osc}\Vert_{Osc_\delta}\lesssim\varep_0.
\end{equation}
\end{proposition}

Proposition \ref{MainBootstrapProp} is the main quantitative result in the paper \cite{IoPaWaWiPoisson} and can be used to easily imply the main result Theorem \ref{thm:main_simple}.

\subsection{Bounds on the nonlinear characteristics}\label{CharcBo} The backward characteristic functions $X$ and $V$ play a key role in the nonlinear analysis. To measure the deviation of the characteristic flow from the free flow we define the functions $\widetilde{Y},\widetilde{W}:\R^3\times\R^3\times\mathcal{I}^2_T\to\R^3$ by
\begin{equation}\label{Lan6}
\begin{split}
\widetilde{Y}(x,v,s,t)&:=X(x+tv,v,s,t)-x-sv,\\
\widetilde{W}(x,v,s,t)&:=V(x+tv,v,s,t)-v.
\end{split}
\end{equation}
The definitions \eqref{cunu2} show that $\widetilde{Y}(x,v,t,t)=0,\widetilde{W}(x,v,t,t)=0$ and
\begin{equation}\label{Lan7}
\begin{split}
\widetilde{W}(x,v,s,t)&=-\int_s^tE(x+\tau v+\widetilde{Y}(x,v,\tau,t),\tau)\,d\tau,\\
\widetilde{Y}(x,v,s,t)&=\int_s^t(\tau-s)E(x+\tau v+\widetilde{Y}(x,v,\tau,t),\tau)\,d\tau.
\end{split}
\end{equation}
Moreover, for any $s, t\in[0,T]$ and $x,v\in\R^3$ we have
\begin{equation}\label{Lan7.5}
\partial_s\widetilde{Y}(x,v,s,t)=\widetilde{W}(x,v,s,t),\qquad \widetilde{Y}(x,v,s,t)=-\int_s^t\widetilde{W}(x,v,\tau,t)\,d\tau.
\end{equation}

The decomposition $\rho=\rho^{stat}+\Re\{e^{-it}\rho^{osc}\}$ in Lemma \ref{rhodeco} induces a natural decomposition of the electric field
\begin{equation}\label{Laga5}
\begin{split}
&E(t)=E^{stat}(t)+\Re\{e^{-it} E^{osc}(t)\},\\
&E^{stat}(t):=(\nabla_x\Delta_x^{-1}\rho^{stat})(t),\qquad E^{osc}(t):=(\nabla_x\Delta_x^{-1}\rho^{osc})(t).
\end{split}
\end{equation}
The bootstrap assumptions \eqref{YW12} induce natural bounds on the electric field components $E^{stat}$ and $E^{osc}$, of the form \begin{equation}\label{Laga11}
\|E^{stat}(t)\|_{L^\infty}\lesssim\varep_1\langle t\rangle^{-3+2\delta},\qquad \|\nabla_xE^{stat}(t)\|_{L^\infty}\lesssim \varep_1\langle t\rangle^{-4+2\delta}\ln(2+t),
\end{equation}
and
\begin{equation}\label{Laga12}
\|E^{osc}(t)\|_{L^\infty}\lesssim \varep_1\langle t\rangle^{-2+\delta},\qquad \|\nabla_{x,t}E^{osc}(t)\|_{L^\infty}\lesssim \varep_1\langle t\rangle^{-3+2\delta},
\end{equation}
for any $t\in[0,T]$. The identities \eqref{Lan7}--\eqref{Lan7.5} can then be used to prove the following pointwise bounds on the functions $\widetilde{Y}$ and $\widetilde{W}$ and their derivatives: 

\begin{lemma}[\protect{\cite[Lemma 3.6]{IoPaWaWiPoisson}}]\label{derivativeschar}
For any $(s,t)\in\mathcal{I}^2_T$ and $x,v\in \R^3$ we have
\begin{equation}\label{cui6}
\begin{split}
|\widetilde{Y}(x,v,s,t)|&\lesssim    \varep_1\min\{\langle s\rangle^{-1+2\delta}\langle v\rangle, \langle s\rangle^{-1/6}\},\\
|\widetilde{W}(x,v,s,t)|&\lesssim    \varep_1\min\{\langle s\rangle^{-2+2\delta}\langle v\rangle, \langle s\rangle^{-7/6}\},
\end{split}
\end{equation}
\begin{equation}\label{nov28eqn2}
\begin{split}
|\partial_t \widetilde{Y}(x,v,s,t)-(t-s) E( x+tv,t)|   &\lesssim \varepsilon_1\langle t-s\rangle \langle t\rangle^{-2+\delta}\langle s \rangle^{-1+1.1\delta},\\
|\partial_t\widetilde{W}(x,v,s,t)+E( x+tv,t)| &\lesssim \varepsilon_1\langle t-s\rangle \langle t\rangle^{-2+\delta}\langle s \rangle^{-2+1.1\delta},
\end{split}
\end{equation}
\begin{equation}\label{cui7}
\begin{split}
|\nabla_x\widetilde{Y}(x,v,s,t)|&\lesssim \varep_1\min\{\langle s\rangle^{-2+2.1\delta}\langle v\rangle, \langle s\rangle^{-7/6}\},\\
|\nabla_x\widetilde{W}(x,v,s,t)|&\lesssim \varep_1\min\{\langle s\rangle^{-3+2.1\delta}\langle v\rangle, \langle s\rangle^{-13/6}\},
\end{split}
\end{equation}
\begin{equation}\label{cui7.5}
\begin{split}
|\nabla_v\widetilde{Y}(x,v,s,t)|&\lesssim \varep_1\min\{\langle s\rangle^{-1+2.1\delta}\langle v\rangle, \langle s\rangle^{-1/6}\},\\
|\nabla_v\widetilde{W}(x,v,s,t)|&\lesssim \varep_1\min\{\langle s\rangle^{-2+2.1\delta}\langle v\rangle, \langle s\rangle^{-7/6}\}.
\end{split}
\end{equation}
\end{lemma}

\subsection{Contributions of the initial data} The components $\rho^{stat}_{1,I}, \rho^{stat}_{1,II}, \rho^{osc}_{1,I}, \rho^{osc}_{1,II}$ can be bounded easily, since they are defined explicitly in terms of the initial data (see the formula \eqref{qwp1}), using only the initial-data assumptions \eqref{bootinit} and some of the bounds in Lemma \ref{derivativeschar} corresponding to $s=0$.

\subsection{Contributions of the reaction term} The bounds on the components $\rho^{stat}_{2,I}$, $\rho^{stat}_{2,II}$, $\rho^{osc}_{2,I}$, $\rho^{osc}_{2,II}$ are more complicated. 

To illustrate the argument, consider for example the component $\rho^{osc}_{2,II}$. Since 
\begin{equation*}
\begin{split}
&X(x+(t-s)v,v,s,t)=\widetilde{Y}(x-sv,v,s,t)+x,\\
&V(x+(t-s)v,v,s,t)=\widetilde{W}(x-sv,v,s,t)+v,
\end{split}
\end{equation*}
it follows from \eqref{qwp1} that
\begin{equation*}
\begin{split}
(&\partial_tL_{2,j})(x,v,s,t)\\
&=\big[-(\partial_iE^l)(x+\widetilde{Y}(x-sv,v,s,t),s) (\partial_t\widetilde{Y}^i)(x-sv,v,s,t) \partial_l M_1(v+\widetilde{W}(x-sv,v,s,t))\\
&-E_l(x+\widetilde{Y}(x-sv,v,s,t),s)  (\partial_t\widetilde{W}^i)(x-sv,v,s,t) \partial_i \partial_{l}M_1(v+\widetilde{W}(x-sv,v,s,t))\big]\widetilde{\varphi}_j(v).
\end{split}
\end{equation*}
Using the formulas \eqref{Lan7} we have
\begin{equation*}
\begin{split}
&(\partial_t\widetilde{Y})(x-sv,v,s,t)=(t-s) E( x+(t-s)v,t)+{\textrm{Error terms}}\\
&(\partial_t\widetilde{W})(x-sv,v,s,t)=-E( x+(t-s)v,t)+{\textrm{Error terms}}.
\end{split}
\end{equation*}
We combine these identities to see that
\begin{equation*}
\begin{split}
&(\partial_tL_{2,j})(x,v,s,t)\approx-(\partial_iE^l)(y-(t-s)v+\widetilde{Y}(y-tv,v,s,t),s)(t-s)E^i(y,t) \\
&\qquad\qquad\qquad\qquad\qquad\times\widetilde{\varphi}_j(v)\partial_{l}M_1(v+\widetilde{W}(y-tv,v,s,t))\\
&+ E^l(y-(t-s)v+\widetilde{Y}(y-tv,v,s,t),s)E^i(y,t)  \cdot\widetilde{\varphi}_j(v)\partial_i \partial_{l}M_1 (v+\widetilde{W}(y-tv,v,s,t)),
\end{split}
\end{equation*}
up to error terms. Therefore, we need to estimate bilinear expressions involving two copies of the electric field $E$ and its derivatives. It turns out that such expressions appear repeatedly in the analysis of the reaction terms.

\subsubsection{Trilinear estimates} Many of the terms arising in the analysis of the reaction term, including the component $\rho^{osc}_{2,II}$ discussed above and many error terms, fit into a general framework of trilinear operators. We define a sufficiently general class here and state some of the relevant estimates. 

Assume that $\theta_1,\theta_2 \in [0, 1]$ and $s\in[0,T]$, and define the \emph{trilinear operators} 
\begin{equation}\label{ropi1}
\begin{split}
&  \mathcal{Q}_{j,k}(f, g; C)(x,\gamma, \tau, s):=\int_{\R^3} \int_{\R^3}\mathcal{K}_{j,k}(x-y, v,\tau,s) C(x, y, v, \gamma, \tau, s)\\
&\times f(y-(s-\tau)v +\theta_1\widetilde{Y}(y-sv, v,\tau,s),\tau)g(y-(s-\gamma)v + \theta_2\widetilde{Y}(y-sv,v,\gamma,s), \gamma)\,dy d v, 
\end{split}
\end{equation}
where $\gamma,\tau\in[0,s]$, $(j,k)\in\Z_+\times\Z$, and the kernel $\mathcal{K}_{j,k}$ satisfies the uniform estimates 
 \begin{equation}\label{june16eqn4}
\big| \mathcal{K}_{j,k}(y, v,\tau,s)\big|\lesssim 2^{3k}(1+2^k|y|)^{-8}\cdot 2^{-3j}\widetilde{\varphi}_{[j-4,j+4]}(v).
\end{equation}

We assume that the coefficient $C$ is differentiable in $\gamma$, and define
\begin{equation}\label{june16eqn3}
\begin{split}
m(C)(\gamma,\tau, s)&:= \|C(x, y, v,\gamma,\tau,s)\|_{L^\infty_{x,y,v}}  + \|\langle\gamma\rangle\partial_\gamma C(x, y, v,\gamma,\tau,s)\|_{L^\infty_{x,y,v}}.
\end{split}
\end{equation}

Assume that $m_2\geq 0$, $t_3,t_4\in [2^{m_2}-1,2^{m_2+1}]\cap [0,s]$, and define the integrated trilinear operators
\begin{equation}\label{june16eqn1}
\mathcal{B}^{ }_{j,k}(f, g; C)(x, \tau, s) = \int_{t_3}^{t_4}\mathcal{Q}_{j,k}(f, g; C)(x, \gamma, \tau, s)\,d\gamma.
\end{equation}

\begin{lemma}[Trilinear estimates \protect{\cite[Lemma 5.1]{IoPaWaWiPoisson}}]\label{keybilinearlemma}
Let  $s\in [0, T]$,  $\tau \in [0, s],$ $2^{m_2-1}\leq s$. With the assumptions of the bootstrap proposition \ref{MainBootstrapProp}, and the notation and assumptions above,   then
\begin{align}
\| \mathcal{B}_{j,k}(\nabla E,  E; C) (., \tau, s)  \|_{B^0_s}&\lesssim   \varep_1^2 \min\{2^{- 1.1 m_2}, \langle \tau \rangle^{-1.1} \}  m^\ast(C)(\tau,s),\label{june16eqn31A1}\\
\| \mathcal{B}_{j,k}(E, \nabla E; C)  (., \tau, s)\|_{B^0_s}&\lesssim \varep_1^2 \min\{2^{- 1.1 m_2}, \langle \tau \rangle^{-1.1} \}  m^\ast(C)(\tau,s),\label{june16eqn31A2}\\
\| \mathcal{B}_{j,k}(E,  E; C) (., \tau, s)\|_{B^0_s}  &\lesssim \varep_1^2 \min\{2^{- 0.1 m_2}, \langle \tau \rangle^{-0.1} \}  m^\ast(C)(\tau,s),\label{june16eqn32}\\
\| \mathcal{B}_{j,k}(\nabla E, \nabla E; C) (., \tau, s)\|_{B^0_s}&\lesssim \varep_1^2  \min\{2^{- 2.1 m_2}, \langle \tau \rangle^{-2.1} \}  m^\ast(C)(\tau,s),\label{june16eqn33}
\end{align}
where $m^\ast(C)(\tau,s):=\sup_{\gamma\in [t_3,t_4]}m(C)(\gamma, \tau, s)$.
\end{lemma}

To prove these bounds we use \eqref{Laga5} to decompose the electric field in the second position (the function $g$) into static and oscillatory components. The point is that the static component $E^{stat}$ satisfies strong decay bounds like \eqref{Laga11}, which lead to the desired conclusions upon integration in $\gamma$. On the other hand, to estimate the contributions of the oscillatory component $E^{osc}$ we can integrate by parts in time in $\gamma$.

\end{document}